\def\year{2019}\relax
\documentclass[letterpaper]{article}
\usepackage{}
\usepackage{amsfonts}
\usepackage{aaai19}  
\usepackage{times}  
\usepackage{helvet}  
\usepackage{courier}  
\usepackage{url}  
\usepackage{graphicx}  
\usepackage{appendix}
\usepackage{amsthm}
\usepackage{amsmath}

\usepackage{multirow}

\usepackage{algorithm}
\usepackage{algorithmic}

\newtheorem{theorem}{Theorem}
\newtheorem{lemma}{Lemma}

\newtheorem{definition}{Definition}
\newtheorem{assumption}{Assumption}
\newtheorem{remark}{Remark}

\usepackage{hyperref}

\usepackage{color}
\usepackage{natbib}

\usepackage{graphicx}
\usepackage[tight,footnotesize]{subfigure}

\begin{document}
%
\title{ Faster Gradient-Free Proximal Stochastic Methods for \\ Nonconvex Nonsmooth Optimization }

\author{Feihu Huang$^{1,2}$, Bin Gu$^{3}$, Zhouyuan Huo$^2$, Songcan Chen$^{1}$\thanks{Corresponding Author.}, Heng Huang$^{2,3}$\\
$^1$College of Computer Science \& Technology, Nanjing University of Aeronautics and Astronautics, Nanjing, 211106, China \\
$^2$Department of Electrical \& Computer Engineering, University of Pittsburgh, PA 15261, USA \quad
$^3$JDDGlobal.com\\
huangfeihu@nuaa.edu.cn, gubin3@jd.com,  zhouyuan.huo@pitt.edu, s.chen@nuaa.edu.cn, heng.huang@pitt.edu
}

\maketitle
\begin{abstract}
Proximal gradient method has been playing an important role to solve many machine learning tasks,
especially for the nonsmooth problems.
However, in some machine learning problems such as the bandit model and the black-box learning problem,
proximal gradient method could fail because the explicit gradients of these problems are difficult
or infeasible to obtain. The gradient-free (zeroth-order) method can address these problems because only the objective function values are required in the optimization.
Recently, the first zeroth-order proximal stochastic
algorithm was proposed to solve the nonconvex nonsmooth problems.
However, its convergence rate is $O(\frac{1}{\sqrt{T}})$ for the nonconvex problems,
which is significantly slower than the best convergence rate $O(\frac{1}{T})$ of the zeroth-order stochastic algorithm,
where $T$ is the iteration number.
To fill this gap, in the paper, we propose a class of faster zeroth-order proximal stochastic methods
with the variance reduction techniques of SVRG and SAGA,
which are denoted as ZO-ProxSVRG and ZO-ProxSAGA, respectively.
In theoretical analysis, we address the main challenge that
an unbiased estimate of the true gradient does not hold in the zeroth-order case,
which was required in previous theoretical analysis of both SVRG and SAGA.
Moreover,  we prove that both ZO-ProxSVRG and ZO-ProxSAGA algorithms have
$O(\frac{1}{T})$ convergence rates.
Finally, the experimental results
verify that our algorithms have
a faster convergence rate than the existing zeroth-order proximal stochastic algorithm.
\end{abstract}

\section{Introduction}
Proximal gradient (PG) methods \citep{Mine1981A,nesterov2004introductory,parikh2014proximal} are a class of powerful optimization tools
in artificial intelligence and machine learning.
In general, it considers the following nonsmooth optimization problem:
\begin{align}
\min_{x\in \mathbb{R}^d} f(x) + \psi(x),
\end{align}
where $f(x)$ usually is the loss function such as hinge loss and logistic loss, and $\psi(x)$ is the nonsmooth structure regularizer such as $\ell_1$-norm regularization.
In recent research, \cite{beck2009fast,nesterov2013gradient} proposed the accelerate PG methods to solve convex problems
by using the Nesterov's accelerated technique. After that, \cite{li2015accelerated} presented a class of
accelerated PG methods for nonconvex optimization.
More recently, \cite{gu2018inexact} introduced inexact PG methods for nonconvex nonsmooth optimization.
To solve the big data problems, the incremental or stochastic PG methods \citep{bertsekas2011incrementalp,xiao2014proximal}
were developed for large-scale convex optimization.
Correspondingly, \cite{ghadimi2016mini,Reddi2016Prox} proposed the stochastic PG methods for large-scale nonconvex optimization.

However, in many machine learning problems,
the explicit expressions of gradients are difficult or infeasible to obtain.
For example, in some complex graphical model inference \citep{wainwright2008graphical} and structure prediction problems \citep{Sokolov2018Sparse},
it is difficult to compute the explicit gradients of the objective functions.
Even worse, in bandit \citep{shamir2017optimal} and black-box learning \citep{chen2017zoo} problems,
only the objective function values are available (the explicit gradients cannot be calculated).
Clearly, the above PG methods will fail in dealing with these scenarios. The gradient-free (zeroth-order) optimization method \citep{Nesterov2017RandomGM} is a promising
choice to address these problems because it only uses the function values in optimization process.
Thus, the gradient-free optimization methods have been increasingly embraced for solving many machine learning problems \citep{conn2009introduction}.

\begin{table*}
  \centering
  \caption{ Comparison of representative zeroth-order stochastic algorithms for finding an $\epsilon$-approximate stationary point of nonconvex problem, i.e.,
   $\mathbb{E}\|\nabla f(x)\|^2\leq \epsilon$
    or $\mathbb{E}\|g_{\eta}(x)\|^2\leq \epsilon$. (S, NS, C and NC are the abbreviations of smooth, nonsmooth, convex and
    nonconvex, respectively. $T$ is the whole iteration number, $d$ is the dimension of data and $n$ denotes the sample size.)
    $B(\leq n)$ is a mini-batch size. }
  \label{tab:1}
\begin{tabular}{c|c|c|c|c}
  \hline
  \textbf{Algorithm} & \textbf{Reference} & \textbf{Gradient estimator} &\textbf{Problem}  & \textbf{Convergence rate}  \\ \hline
  RSGF & \citet{Ghadimi2013StochasticFA} & GauSGE &S(NC)  &  $O(\sqrt{\frac{d}{T}})$ \\ \hline
  ZO-SVRG & \citet{liu2018zeroth} & CooSGE & S(NC)  & $O(\frac{d}{T})$ \\ \hline
  \multirow{2}*{SZVR-G}  & \multirow{2}*{\citet{Liu2018StochasticZO}} & GauSGE  & S(NC) & $O(\max(d^{\frac{2}{3}}B^{\frac{1}{3}},d^{\frac{1}{3}}B^{\frac{2}{3}})/T)$ \\ \cline{3-5}
    &  & GauSGE & NS(NC) & $O(d^{\frac{5}{\sqrt{33}}}B^{\frac{1}{\sqrt{33}}}/T^{\sqrt{\frac{3}{11}}})$ \\ \hline
  RSPGF  &\citet{ghadimi2016mini} & GauSGE & S(NC) + NS(C)  & $O(\sqrt{\frac{d}{T}})$ \\ \hline
  \multirow{2}*{ZO-ProxSVRG} & \multirow{2}*{Ours}  & CooSGE &S(NC) + NS(C)  & $O(\frac{d}{T})$  \\ \cline{3-5}
   &  & GauSGE & S(NC) + NS(C)   & $O(\frac{d}{T} + d\sigma^2)$ \\ \hline
  \multirow{2}*{ZO-ProxSAGA}  & \multirow{2}*{Ours}  & CooSGE &S(NC) + NS(C)   & $O(\frac{d}{T})$  \\ \cline{3-5}
   &  & GauSGE &S(NC) + NS(C)   & $O(\frac{d}{T}+d\sigma^2)$  \\
  \hline
\end{tabular}
\end{table*}

Although many gradient-free methods have recently been developed and studied
\citep{agarwal2010optimal,Nesterov2017RandomGM,liu2018admm}, they often
suffer from the high variances of zeroth-order gradient estimates.
In addition, these algorithms are mainly designed for smooth or convex settings,
which will be discussed in the below related works,
thus limiting their applicability in
a wide range of nonconvex nonsmooth machine learning problems such as involving the nonconvex loss functions and
nonsmooth regularization.

In this paper, thus, we propose a class of faster gradient-free proximal stochastic methods
for solving the nonconvex nonsmooth problem as follows:
\begin{align} \label{eq:1}
\min_{x \in \mathbb{R}^d} F(x) =: f(x) + \psi(x), \ f(x)=:\frac{1}{n}\sum_{i=1}^n f_i(x)
\end{align}
where each $f_i(x)$ is a \emph{nonconvex} and smooth loss function, and $\psi(x)$ is
a convex and \emph{nonsmooth} regularization term.
Until now, there are few zeroth-order stochastic methods for solving the problem \eqref{eq:1}
except a recent attempt proposed in \citep{ghadimi2016mini}. Specifically,
\cite{ghadimi2016mini} have proposed a randomized stochastic projected gradient-free method
(RSPGF), \emph{i.e.}, a zeroth-order proximal stochastic gradient method.
However, due to the large variance of zeroth-order estimated gradient generated from randomly selecting the sample and the direction of derivative, the RSPGE only has a convergence rate $O(\frac{1}{\sqrt{T}})$ ,
which is significantly slower than $O(\frac{1}{T})$, the best
convergence rate of the zeroth-order stochastic algorithm.
To accelerate the RSPGF algorithm, we use the variance reduction strategies in the first-order methods,
\emph{i.e.}, SVRG \citep{xiao2014proximal} and SAGA \citep{defazio2014saga}, to reduce the variance of estimated stochastic gradient.

Although SVRG and SAGA have shown good performances, applying these strategies to the zeroth-order method is \textbf{not a trivial
task}. The main challenge arises due to that both SVRG and SAGA rely on the assumption that a stochastic
gradient is an \textbf{unbiased} estimate of the true full gradient. However, it does not hold in
the zeroth-order algorithms. In the paper, thus, we will fill this gap between
zeroth-order proximal stochastic method and the classic variance reduction approaches (SVRG and SAGA).

\subsection{Main Contributions}
In summary, our main contributions are summarized as follows:
\begin{itemize}
\item We propose a class of faster gradient-free proximal stochastic methods (ZO-ProxSVRG and ZO-ProxSAGA),
      based on the variance reduction techniques of SVRG and SAGA. Our new algorithms only use the objective function values in the optimization process.
\item Moreover, we provide the theoretical analysis on the convergence properties of both new ZO-ProxSVRG and ZO-ProxSAGA methods.
      Table \ref{tab:1} shows the specifical convergence rates of the proposed algorithms and other related ones. In particular, our algorithms have faster convergence rate $O(\frac{1}{T})$ than
      $O(\frac{1}{\sqrt{T}})$ of the RSPGF \citep{ghadimi2016mini} (the existing stochastic PG algorithm for solving nonconvex nonsmoothing problems).
\item Extensive experimental results and theoretical analysis demonstrate the effectiveness of our algorithms.
\end{itemize}

\section{Related Works}
Gradient-free (zeroth-order) methods have been effectively used to solve many machine learning problems,
where the explicit gradient is difficult or infeasible to obtain, and have also been widely studied.
For example, \cite{Nesterov2017RandomGM} proposed  several random gradient-free methods by using Gaussian smoothing technique.
\cite{duchi2015optimal} proposed a zeroth-order mirror descent algorithm.
More recently, \cite{yu2018generic,dvurechensky2018accelerated} presented the accelerated
zeroth-order methods for the convex optimization.
To solve the nonsmooth problems, the zeroth-order online or stochastic ADMM methods \citep{liu2018admm,gao2018information} have been introduced.

The above zeroth-order methods mainly focus on the (strongly) convex problems.
In fact, there exist many nonconvex machine learning tasks, whose explicit gradients are not available,
such as the nonconvex black-box learning problems \citep{chen2017zoo,liu2018zeroth}.
Thus, several recent works have begun to study the zeroth-order stochastic methods for the nonconvex optimization.
For example, \cite{Ghadimi2013StochasticFA} proposed the randomized stochastic gradient-free (RSGF) method,
\emph{i.e.}, a zeroth-order stochastic gradient method.
To accelerate optimization, more recently, \cite{liu2018zeroth,Liu2018StochasticZO} proposed the zeroth-order stochastic
variance reduction gradient (ZO-SVRG) methods.
Moreover, to solve the large-scale machine learning problems, some asynchronous parallel stochastic
zeroth-order algorithms have been proposed in \citep{gu2016zeroth,Lian2016A,Gu2018faster}.

Although the above zeroth-order stochastic methods can effectively solve the nonconvex optimization,
there are few zeroth-order stochastic methods for the \emph{nonconvex nonsmooth}
composite optimization except the RSPGF method presented in \citep{ghadimi2016mini}.
In addition, \cite{Liu2018StochasticZO} have also studied the zeroth-order algorithm for solving the nonconvex nonsmooth
problem, which is different from problem \eqref{eq:1}.

\section{Zeroth-Order Proximal Stochastic Method Revisit}
In this section, we briefly review the zeroth-order proximal stochastic gradient (ZO-ProxSGD) method
to solve the problem  \eqref{eq:1}.
Before that, we first revisit the proximal gradient descent (ProxGD) method \citep{Mine1981A}.

ProxGD is an effective method to solve the problem \eqref{eq:1} via the following iteration:
\begin{align}
 x_{t+1} = \mbox{Prox}_{\eta \psi} \big(x_t - \eta \nabla f(x_t)\big), \ t=0,1,\cdots,
\end{align}
where $\eta>0$ is a step size, and $\mbox{Prox}_{\eta \psi}(\cdot)$ is a \emph{proximal operator} defined as:
\begin{align} \label{eq:3}
 \mbox{Prox}_{\eta \psi}(x) = \mathop{\arg\min}_{y\in \mathbb{R}^d} \big\{ \psi(y) + \frac{1}{2\eta}\|y-x\|^2 \big\}.
\end{align}

As discussed above, because ProxGD needs to compute the gradient at each iteration,
it cannot be applied to solve the problems, where the explicit gradient of function $f(x)$ is not available.
For example, in the black-box machine learning model, only function values (\emph{e.g.},
prediction results) are available \cite{chen2017zoo}. To avoid
computing explicit gradient, we use the zeroth-order gradient estimators \citep{Nesterov2017RandomGM,liu2018zeroth} to
estimate the gradient only by function values.

\begin{itemize}
\item Specifically, we use the \textbf{Gau}ssian \textbf{S}moothing \textbf{G}radient \textbf{E}stimator
(\textbf{GauSGE}) \citep{Nesterov2017RandomGM,ghadimi2016mini} to estimate the
 gradients as follows:
 \begin{align} \label{eq:4}
  \hat{\nabla} f_i(x) = \frac{f_i(x+\mu u_{i}) - f_i(x)}{\mu}u_{i},\;\;i\in [n]\,,
 \end{align}
 where $\mu$ is a smoothing parameter, and $\{u_{i}\}_{i=1}^n$ denote
 \emph{i.i.d.} random directions drawn from a zero-mean isotropic
 multivariate Gaussian distribution $\mathcal{N}(0,I)$.

\item Moreover, to obtain better estimated gradient,
  we can use the \textbf{Coo}rdinate \textbf{S}moothing \textbf{G}radient \textbf{E}stimator (\textbf{CooSGE})
 \citep{gu2016zeroth,Gu2018faster,liu2018zeroth} to estimate the gradients as
 follows:
 \begin{align} \label{eq:5}
  \hat{\nabla} f_i(x) = \sum_{j=1}^d \frac{f_i(x+\mu_j e_j)-f_i(x-\mu_j e_j)}{2\mu_j}e_j,\;\;i\in [n]\,,
 \end{align}
 where $\mu_j$ is a coordinate-wise smoothing parameter, and $e_j$ is
 a standard basis vector with $1$ at its $j$-th coordinate, and $0$ otherwise.
 Although the CooSGE need more function queries than the GauSGE, it can get better estimated gradient, and
 even can make the algorithms to obtain a faster convergence rate.
\end{itemize}

Finally, based on these estimated gradients, we give a zeroth-order proximal
gradient descent (ZO-ProxGD) method,
which performs the following iteration:
\begin{align}
 x_{t+1} = \mbox{Prox}_{\eta \psi} \big(x_t - \eta \hat{\nabla} f(x_t)\big), \ t=0,1,\cdots,
\end{align}
where $\hat{\nabla} f(x) = \frac{1}{n}\sum_{i=1}^n \hat{\nabla} f_i(x)$.

Since ZO-ProxGD needs to estimate full gradient $\hat{\nabla} f(x)=\frac{1}{n}\sum_{i=1}^n \nabla f_i(x)$,
when $n$ is large in the problem \eqref{eq:1}, its high cost per iteration is prohibitive.
As a result, \cite{ghadimi2016mini} proposed the RSPGF (\emph{i.e.}, ZO-ProxSGD)
with performing the following iteration:
\begin{align}
 x_{t+1} = \mbox{Prox}_{\eta \psi} \big(x_t - \eta \hat{\nabla} f_{\mathcal{I}_t}(x_t)\big), \ t=0,1,\cdots,
\end{align}
where $\hat{\nabla} f_{\mathcal{I}_t}(x_t) = \frac{1}{b}\sum_{i\in \mathcal{I}_t} \hat{\nabla} f_i(x)$, $\mathcal{I}_t \in \{1,2,\cdots,n\}$
and $b=|\mathcal{I}_t|$ is the mini-batch size.

\section{New Faster Zeroth-Order Proximal Stochastic Methods }
In this section, to efficiently solve the large-scale nonconvex nonsmooth problems, we propose
a class of faster zeroth-order proximal stochastic methods with the variance reduction (VR) techniques of SVRG and SAGA, respectively.

\subsection{ZO-ProxSVRG }
In the subsection, we propose the zeroth-order proximal SVRG (ZO-ProxSVRG) method by using
VR technique of SVRG in \citep{xiao2014proximal,Reddi2016Prox}.

The corresponding algorithmic framework is described in Algorithm \ref{alg:1},
where we use a mixture stochastic gradient
$\hat{v}_{t}^{s} = \hat{\nabla}
f_{\mathcal{I}_t}(x_{t}^{s})-\hat{\nabla} f_{\mathcal{I}_t}(\tilde{x}^{s}) + \hat{\nabla} f(\tilde{x}^{s})$.
Note that $\mathbb{E}_{\mathcal{I}_t}[\hat{v}_{t}^{s}] = \hat{\nabla}
f(x_{t}^{s})\neq \nabla f(x_{t}^{s})$, \emph{i.e.}, this stochastic gradient
is a \textbf{biased} estimate of the true full gradient.
Although the SVRG has shown a great promise,
it relies upon the assumption that the stochastic
gradient is an \textbf{unbiased }estimate of the true full gradient.
Thus, adapting the similar ideas of SVRG to zeroth-order optimization
is not a trivial task. To address this issue, we analyze the upper bound for the variance of
the estimated gradient $\hat{v}_{t}^{s}$, and choose the appropriate step size $\eta$ and
smoothing parameter $\mu$ to control this variance,
which will be in detail discussed in the below theorems.

Next, we derive the upper bounds
for the variance of estimated gradient $\hat{v}_{t}^{s}$ based on the CooSGE and the GauSGE, respectively.

\begin{lemma} \label{lem:1}
In Algorithm \ref{alg:1} using the CooSGE, given the mixture estimated gradient
$\hat{v}^s_t = \hat{\nabla} f_{\mathcal{I}_t}(x_{t}^{s})-\hat{\nabla} f_{\mathcal{I}_t}(\tilde{x}^s)+\hat{\nabla} f(\tilde{x}^s)$,
then the following inequality holds
\begin{align}
 \mathbb{E}\|\hat{v}^s_t-\nabla f(x^s_t)\|^2 \leq \frac{2\delta_n L^2
 d}{b} \mathbb{E}\|x^s_t-\tilde{x}^s\|^2 + \frac{L^2 d^2\mu^2}{2},
\end{align}
where $0 \leq \delta_n \leq 1$.
\end{lemma}

\begin{remark}
 Lemma \ref{lem:1} shows that variance of $\hat{v}^s_t$ has an upper bound.
 As the number of iterations increases, both $x_t^{s}$ and $\tilde{x}^s$ will approach the same stationary point $x^*$,
 then the variance of stochastic gradient decreases, but does not vanishes, due to using the
 zeroth-order estimated gradient.
\end{remark}

\begin{lemma} \label{lem:2}
In Algorithm \ref{alg:1} using the GauSGE, given the estimated gradient
$\hat{v}^s_t = \hat{\nabla} f_{\mathcal{I}_t}(x_{t}^{s})-\hat{\nabla} f_{\mathcal{I}_t}(\tilde{x}^s)+\hat{\nabla} f(\tilde{x}^s)$,
then the following inequality holds
\begin{align}
 & \mathbb{E}\|\hat{v}^s_t-\nabla f(x_t)\|^2 \leq (2+\frac{12\delta_n}{b})(d+6)^3 L^2\mu^2
    \nonumber \\
 & \quad + \frac{6\delta_nL^2}{b}\mathbb{E}\|x^s_t-\tilde{x}^s\|^2+ (4 + \frac{24\delta_n}{b})(2d+9)\sigma^2.
\end{align}
\end{lemma}

\begin{remark}
 Lemma \ref{lem:2} shows that variance of $\hat{v}^s_t$ has an upper bound.
 As the number of iterations increases, both $x_t^{s}$ and $\tilde{x}^s$ will approach the same stationary point $x^*$,
 then the variance of stochastic gradient decreases.
\end{remark}

\begin{algorithm}[htb]
   \caption{ ZO-ProxSVRG for Nonconvex Optimization  }
   \label{alg:1}
\begin{algorithmic}[1]
   \STATE {\bfseries Input:} mini-batch size $b$, $S$, $m$ and step size $\eta>0$;
   \STATE {\bfseries Initialize:} $x_0^1=\tilde{x}^1\in \mathbb{R}^d$;
   \FOR {$s=1,2,\cdots,S$}
   \STATE{} $\hat{\nabla }f(\tilde{x}^{s})=\frac{1}{n}\sum_{i=1}^n\hat{\nabla} f_i(\tilde{x}^{s})$;
   \FOR {$t=0,1,\cdots,m-1$}
   \STATE{} Uniformly randomly pick a mini-batch $\mathcal{I}_t \subseteq \{1,2,\cdots,n\}$ such that $|\mathcal{I}_t|=b$;
   \STATE{} Using \eqref{eq:4} or \eqref{eq:5} to estimate mixture stochastic gradient $\hat{v}_{t}^{s} = \hat{\nabla} f_{\mathcal{I}_t}(x_{t}^{s})
            -\hat{\nabla} f_{\mathcal{I}_t}(\tilde{x}^{s}) + \hat{\nabla} f(\tilde{x}^{s})$;
   \STATE{} $x^{s}_{t+1} = \mbox{Prox}_{\eta \psi} (x_t^{s} - \eta \hat{v}_{t}^{s})$;
   \ENDFOR
   \STATE{} $\tilde{x}^{s+1}= x_m^{s}$ and $x_0^{s+1}=x_{m}^{s}$;
   \ENDFOR
   \STATE {\bfseries Output:} Iterate $x$ chosen uniformly random from $\{(x_{t}^s)_{t=1}^{m}\}_{s=1}^S$.
\end{algorithmic}
\end{algorithm}

\subsection{ZO-ProxSAGA }
In the subsection, we propose the zeroth-order proximal SAGA (ZO-ProxSAGA) method via using
VR technique of SAGA in \citep{defazio2014saga,Reddi2016Prox}.

The corresponding algorithmic description is given in Algorithm \ref{alg:2},
where we use a mixture stochatic gradient $\hat{v}_{t} = \frac{1}{b} \sum_{i_t\in \mathcal{I}_t}
\big(\hat{\nabla} f_{i_t}(x_{t})-\nabla f_{i_t}(z^t_{i_t}) \big) + \hat{\phi}_t$.
Similarly, $\mathbb{E}_{\mathcal{I}_t}[\hat{v}_{t}] = \hat{\nabla} f(x_{t}^{s})\neq \nabla f(x_{t}^{s})$, \emph{i.e.},
this stochastic gradient is a \textbf{biased} estimate of the true full gradient.
Note that in Algorithm \ref{alg:2}, due to $\sum_{i_t\in \mathcal{I}_t}\hat{\nabla} f_{i_t}(z^{t+1}_{i_t})
=\sum_{i_t\in \mathcal{I}_t} \hat{\nabla} f_{i_t}(x_{t}) $,
the step 8 can use directly the term $\sum_{i_t\in \mathcal{I}_t} \big(\hat{\nabla} f_{i_t}(x_{t})-\hat{\nabla} f_{i_t}(z^t_{i_t}) \big)$,
which is computed in the step 5, to avoid unnecessary calculations.
Next, we give the upper bounds for the variance of stochastic gradient $\hat{v}_{t}$
based on the CooSGE and the GauSGE, respectively.

\begin{lemma} \label{lem:3}
In Algorithm \ref{alg:2} using the CooSGE, given the estimated gradient
$\hat{v}_t = \frac{1}{b} \sum_{i_t\in \mathcal{I}_t} \big(\hat{\nabla} f_{i_t}(x_{t})-\hat{\nabla} f_{i_t}(z^t_{i_t}) \big) + \hat{\phi}_t$
with $\hat{\phi}_t = \frac{1}{n} \sum_{i=1}^n \hat{\nabla} f_i(z^t_i)$,
then the following inequality holds
\begin{align}
 \mathbb{E}\|\hat{v}_t-\nabla f(x_t)\|^2 \leq \frac{2 L^2d}{nb} \sum_{i=1}^n \mathbb{E} \|x_t-z^t_i\|^2_2 + \frac{L^2 d^2\mu^2}{2}.
\end{align}
\end{lemma}

\begin{remark}
 Lemma \ref{lem:3} shows that variance of $\hat{v}_t$ has an upper bound.
 As the number of iterations increases, both $x_t$ and $\{z^t_i\}_{i=1}^n$ will approach the same stationary point,
 then the variance of stochastic gradient decreases.
\end{remark}

\begin{lemma} \label{lem:4}
In Algorithm \ref{alg:2} using GauSGE, given the estimated gradient
$\hat{v}_t = \frac{1}{b} \sum_{i_t\in \mathcal{I}_t} \big(\hat{\nabla} f_{i_t}(x_{t})-\hat{\nabla} f_{i_t}(z^t_{i_t}) \big) + \hat{\phi}_t$
with $\hat{\phi}_t = \frac{1}{n} \sum_{i=1}^n \hat{\nabla} f_i(z^t_i)$,
then the following inequality holds
\begin{align}
 & \mathbb{E}\|\hat{v}_t-\nabla f(x_t)\|^2 \leq (2+\frac{12}{b})(d+6)^3 L^2\mu^2 \nonumber \\
 & \quad +\frac{6L^2}{nb}\sum_{i=1}^n\mathbb{E}\|x_t-z^t_i\|^2+ (4+\frac{24}{b})(2d+9)\sigma^2.
\end{align}
\end{lemma}

\begin{remark}
 Lemma \ref{lem:4} shows that variance of $\hat{v}_t$ has an upper bound.
 As the number of iterations increases, both $x_t$ and $\{z^t_i\}_{i=1}^n$ will approach the same stationary point $x^*$,
 then the variance of stochastic gradient decreases.
\end{remark}
\begin{algorithm}[htb]
   \caption{ ZO-ProxSAGA for Nonconvex Optimization }
   \label{alg:2}
\begin{algorithmic}[1]
   \STATE {\bfseries Input:} mini-batch size $b$, $T$ and step size $\eta>0$;
   \STATE {\bfseries Initialize:} $x_0 \in \mathbb{R}^d$, and $z_i^0=x_0$ for $i\in \{1,2,\cdots,n\}$, $\hat{\phi}_0=\frac{1}{n}\sum_{i=1}^n \hat{\nabla} f_i(z^0_i)$;
   \FOR {$t=0,1,\cdots,T-1$}
   \STATE{} Uniformly randomly pick a mini-batch $\mathcal{I}_t \subseteq \{1,2,\cdots,n\}$ (with replacement) such that $|\mathcal{I}_t|=b$;
   \STATE{} Using \eqref{eq:4} or \eqref{eq:5} to estimate mixture stochastic gradient $\hat{v}_{t} = \frac{1}{b} \sum_{i_t\in \mathcal{I}_t} \big(\hat{\nabla} f_{i_t}(x_{t})
            -\hat{\nabla} f_{i_t}(z^t_{i_t}) \big) + \hat{\phi}_t$;
   \STATE{} $x_{t+1} = \mbox{Prox}_{\eta \psi} (x_t - \eta \hat{v}_{t})$;
   \STATE{} $z^{t+1}_{i_t}= x_{t}$ for $i \in \mathcal{I}_t$
             and $z^{t+1}_i=z^t_i$ for $i \notin \mathcal{I}_t$;
   \STATE{} $\hat{\phi}_{t+1}=\hat{\phi}_t-\frac{1}{n}\sum_{i_t\in \mathcal{I}_t}\big(\hat{\nabla} f_{i_t}(z^t_{i_t})-\hat{\nabla} f_{i_t}(z^{t+1}_{i_t})\big)$;
   \ENDFOR
   \STATE {\bfseries Output:} Iterate $x$ chosen uniformly random from $\{x_{t}\}_{t=1}^{T}$.
\end{algorithmic}
\end{algorithm}

\section{Convergence Analysis}
In this section, we conduct the convergence analysis of both ZO-ProxSVRG and ZO-ProxSAGA.
First, we give some mild
assumptions regarding problem \eqref{eq:1} as follows:

\begin{assumption}
For $\forall i \in \{1,2,\cdots,n\}$, gradient of the function $f_i$ is Lipschitz continuous with a Lipschitz constant $L>0$,
such that
\begin{align}
\|\nabla f_i(x)-\nabla f_i(y)\| \leq L \|x - y\|, \ \forall x,y \in \mathbb{R}^d, \nonumber
\end{align}
which implies
\begin{align}
f_i(x) \leq f_i(y) + \nabla f_i(y)^T(x-y) + \frac{L}{2}\|x-y\|^2.  \nonumber
\end{align}
\end{assumption}

\begin{assumption}
 The gradient is bounded as $\|\nabla f_i(x)\|^2 \leq \sigma^2$ for all $i = 1,2,\cdots,n$.
\end{assumption}

The first assumption is standard for the convergence analysis of
the zeroth-order algorithms \citep{ghadimi2016mini,Nesterov2017RandomGM,liu2018zeroth}.
The second assumption gives the bounded gradient used
in \citep{Nesterov2017RandomGM,liu2018admm}, which is relatively stricter
than the bounded variance of gradient in \citep{Lian2016A,liu2018zeroth,Liu2018StochasticZO},
due to that we need to analyze more complex problem \eqref{eq:1} including a non-smooth part. Next, we
introduce the standard \emph{gradient mapping} \citep{parikh2014proximal} used in
the convergence analysis as follows:
\begin{align}
g_{\eta}(x) = \frac{1}{\eta}\big(x - \mbox{Prox}_{\eta \psi}(x-\eta \nabla f(x))\big).
\end{align}
For the nonconvex problems, if $g_{\eta}(x)=0$, the point $x$ is a
critical point \citep{parikh2014proximal}. Thus, we can use the
following definition as the convergence metric.
\begin{definition}
 \citep{Reddi2016Prox} A solution $x$ is called $\epsilon$-accurate, if $\mathbb{E}\|g_{\eta}(x)\|^2\leq \epsilon$ for some $\eta>0$.
\end{definition}

\subsection{Convergence Analysis of ZO-ProxSVRG }
In the subsection, we show the convergence analysis of the ZO-ProxSVRG
with the CooSGE (\textbf{ZO-ProxSVRG-CooSGE})
and the GauSGE (\textbf{ZO-ProxSVRG-GauSGE}), respectively.

\begin{theorem} \label{th:1}
 Assume the sequence $\{(x^s_t)_{t=1}^m\}_{s=1}^S$ generated from Algorithm \ref{alg:1} using the \textbf{CooSGE},
 and define a sequence $\{c_t\}_{t=1}^m$ as follows: for $s=1,2,\cdots,S$
  \begin{equation}
  c_t = \left\{
  \begin{aligned}
   & \frac{\delta_n L^2d\eta}{b} + c_{t+1}(1+\beta), \ 0 \leq t \leq m-1; \\
   & 0, \ t = m
 \end{aligned}
  \right.\end{equation}
 where $\beta>0$. Let $T=mS$, $\eta = \frac{\rho}{dL} \ (0<\rho<\frac{1}{2})$ and $b$ satisfies the following inequality:
 \begin{align}
  \frac{8\rho^2 m^2}{b} +\rho\leq 1,
 \end{align}
 then we have
 \begin{align}
  \mathbb{E} \|g_{\eta}(x^s_t)\|^2 \leq \frac{\mathbb{E} [F(x^1_0)-F(x_*)]}{T\gamma} + \frac{ L^2 d^2\mu^2\eta}{4\gamma},
 \end{align}
 where $\gamma=\frac{\eta}{2}-L\eta^2$ and $x^*$ is an optimal solution of the problem \eqref{eq:1}.
 Further let $b=[n^{\frac{2}{3}}]$, $m=[n^{\frac{1}{3}}]$, $\rho=\frac{1}{4}$ and $\mu=O(\frac{1}{\sqrt{dT}})$,
 we have
 \begin{align}
 \mathbb{E} \|g_{\eta}(x^s_t)\|^2 \leq \frac{ 16d L\mathbb{E} [F(x^1_0) - F(x_*)]}{T} + O(\frac{d}{T}).
 \end{align}
\end{theorem}

\begin{remark}
 Theorem \ref{th:1} shows that, given $\mu=O(\frac{1}{\sqrt{dT}})$, $b=[n^{\frac{2}{3}}]$ and $m=[n^{\frac{1}{3}}]$,
 the ZO-ProxSVRG-CooSGE has $O(\frac{d}{T})$ convergence rate.
\end{remark}

\begin{theorem} \label{th:2}
 Assume the sequence $\{(x^s_t)_{t=1}^m\}_{s=1}^S$ generated from Algorithm \ref{alg:1} using the GauSGE,
 and define a sequence $\{c_t\}_{t=1}^m$ as follows: for $s=1,2,\cdots,S$
  \begin{equation}
  c_t = \left\{
  \begin{aligned}
   & \frac{3\delta_n L^2\eta}{b} + c_{t+1}(1+\beta), \ 0 \leq t \leq m-1; \\
   & 0, \ t = m
 \end{aligned}
  \right.\end{equation}
 where $\beta>0$. Let $\eta = \frac{\rho}{L} \ (0<\rho<\frac{1}{2})$ and $b$ satisfies the following inequality:
 \begin{align}
  \frac{24\rho^2m^2}{b} + \rho\leq 1,
 \end{align}
 then we have
 \begin{align}
  \mathbb{E} \|g_{\eta}(x^s_t)\|^2 \leq & \frac{\mathbb{E}[F(x^1_0)-F(x_*)]}{T\gamma}+ (1+\frac{6\delta_n}{b})(d+6)^3 \frac{L^2\mu^2\eta}{\gamma} \nonumber \\
  & + (2 + \frac{12\delta_n}{b})(2d+9)\frac{\sigma^2\eta}{\gamma},
 \end{align}
 where $\gamma=\frac{\eta}{2}-\eta^2 L$ and $x^*$ is an optimal solution of the problem \eqref{eq:1}.
 Further let $b=[n^{\frac{2}{3}}]$, $m=[n^{\frac{1}{3}}]$, $\rho=\frac{1}{6}$ and $\mu=O(\frac{1}{d\sqrt{T}})$, we have
 \begin{align}
 \mathbb{E} \|g_{\eta}(x^s_t)\|^2 \leq &\frac{18L\mathbb{E}[F(x^1_0) - F(x_*)]}{T}
  + O(\frac{d}{T})  \nonumber \\
 &+ O(d \sigma^2).
 \end{align}
\end{theorem}

\begin{remark}
 Theorem \ref{th:2} shows that given $\mu=O(\frac{1}{d\sqrt{T}})$, $b=[n^{\frac{2}{3}}]$ and $m=[n^{\frac{1}{3}}]$,
 the ZO-ProxSVRG-GauSGE has $O(\frac{d}{T}+d\sigma^2)$ convergence rate, in which the part $O(d\sigma^2)$ generates from the
 GauSGE.
\end{remark}

\subsection{Convergence Analysis of ZO-ProxSAGA }
In this subsection, we provide the convergence analysis of the ZO-ProxSAGA
with the CooSGE (\textbf{ZO-ProxSAGA-CooSGE})
and the GauSGE (\textbf{ZO-ProxSAGA-GauSGE}), respectively.

\begin{theorem} \label{th:3}
 Assume the sequence $\{x_t\}_{t=1}^T$ generated from Algorithm \ref{alg:2} using the CooSGE,
 and define a positive sequence $\{c_t\}_{t=1}^T$ as follows:
 \begin{align}
  c_t = \frac{ L^2d\eta}{b}+c_{t+1}(1-p)(1+\beta)
 \end{align}
 where $\beta>0$.
 Let $c_T=0$, $\eta =\frac{\rho}{Ld} \ (0< \rho < \frac{1}{2})$, and $b$ satisfies the following inequality:
 \begin{align}
  \frac{32\rho^2n^2}{b^3} + \rho \leq 1,
 \end{align}
 then we have
 \begin{align}
  \mathbb{E} \|g_{\eta}(x_t)\|^2 \leq \frac{\mathbb{E}[F(x_0)-F(x_*)]}{T\gamma} + \frac{L^2 d^2\mu^2\eta}{4\gamma},
 \end{align}
 where $\gamma = \frac{\eta}{2} - L\eta^2$ and $x^*$ is an optimal solution of the problem \eqref{eq:1}.
 Further let $b = [n^{\frac{2}{3}}]$, $\rho = \frac{1}{8}$ and $\mu=O(\frac{1}{\sqrt{dT}})$,
 we have
 \begin{align}
 \mathbb{E} \|g_{\eta}(x_t)\|^2 \leq \frac{64dL\mathbb{E}[F(x_0) - F(x_*)]}{3T} + O(\frac{d}{T}).
 \end{align}
\end{theorem}

\begin{remark}
 Theorem \ref{th:3} shows that given $\mu=O(\frac{1}{\sqrt{dT}})$ and $b=[n^{\frac{2}{3}}]$, the ZO-ProxSAGA-CooSGE
 has $O(\frac{d}{T})$ convergence rate.
\end{remark}

\begin{theorem} \label{th:4}
 Assume the sequence $\{x_t\}_{t=1}^T$ generated from Algorithm \ref{alg:2} using the GauSGE,
 and define a positive sequence $\{c_t\}_{t=1}^T$ as follows:
 \begin{align}
  c_t = \frac{ 3L^2\eta}{b}+c_{t+1}(1-p)(1+\beta),
 \end{align}
 where $\beta>0$. Let $c_T=0$, $\eta = \frac{\rho}{L} (0< \rho < \frac{1}{2})$ and $b$ satisfies the following inequality:
 \begin{align}
  \frac{96\rho^2n^2}{b^3} + \rho \leq 1,
 \end{align}
 then we have
 \begin{align}
  \mathbb{E} \|g_{\eta}(x_t)\|^2 \leq & \frac{\mathbb{E}[F(x_0)-F(x_*)]}{T\gamma} + \frac{(2+\frac{12}{b})(2d+9)\sigma^2\eta}{\gamma}\nonumber \\
  & + \frac{(1+\frac{6}{b})(d+6)^3 L^2\mu^2\eta}{\gamma},
 \end{align}
 where $\gamma=\frac{1}{2\eta}-L\eta^2$ and $x^*$ is an optimal solution of the problem \eqref{eq:1}. Further given $b=[n^{\frac{2}{3}}]$,
 $\rho = \frac{1}{12}$ and $\mu=O(\frac{1}{d\sqrt{T}})$, we have
 \begin{align}
 \mathbb{E} \|g_{\eta}(x_t)\|^2 \leq  & \frac{144L\mathbb{E}[F(x_0)-F(x_*)]}{5T}+ O(\frac{d}{T})+ O(d\sigma^2).
 \end{align}
\end{theorem}

\begin{remark}
 Theorem \ref{th:4} shows that given $\mu=O(\frac{1}{d\sqrt{T} })$ and $b=[n^{\frac{2}{3}}]$,
 the ZO-ProxSAGA-GauSGE has $O(\frac{d}{T}+d\sigma^2)$ convergence rate, in which the part $O(d\sigma^2)$ generates from the
 GauSGE.
\end{remark}

\emph{All related proofs are in the supplementary document.}
\section{Experiments}

In this section, we will compare the proposed algorithms
(ZO-ProxSVRG-CooSGE, ZO-ProxSVRG-GauSGE, ZO-ProxSAGA-CooSGE,
ZO-ProxSAGA-GauSGE) with the RSPGF method \citep{ghadimi2016mini}
on two applications: \textbf{black-box binary classification} and \textbf{ adversarial
attacks on black-box deep neural networks (DNNs)}.
Note that the RSPGF  uses the GauSGE to estimate gradient.

\subsection{Black-Box Binary Classification }

\subsubsection{Experimental Setup}
In this experiment, we apply our algorithms to learn the black-box binary classification problem.
Specifically, given a set of training samples $\{a_i,l_i\}_{i=1}^n$,
where $a_i\in \mathbb{R}^d$ and $l_i \in \{-1,1\}$,  we find the
optimal predictor $x\in \mathbb{R}^d$ by solving the following problem:
\begin{align}
\min_{x\in \mathbb{R}^d} \frac{1}{n}\sum_{i=1}^n f_i(x) + \lambda_1\|x\|_1 + \lambda_2\|x\|_2^2,
\end{align}
where $f_i(x)$ is the black-box loss function,
that only returns the function value given an input.
Here, we specify the non-convex \emph{sigmoid loss} function
$f_i(x)=\frac{1}{1+\exp(l_i a_i^Tx)}$ in the black-box setting.

In the experiment, we use
the publicly available real datasets\footnote{\emph{20news} is from the website \url{https://cs.nyu.edu/~roweis/data.html};
\emph{a9a}, \emph{w8a} and \emph{covtype.binary} are from the website \url{www.csie.ntu.edu.tw/~cjlin/libsvmtools/datasets/}.},
which are summarized in Table \ref{tab:2}.
In the algorithms, we fix the mini-batch size $b=20$,
the smoothing parameters $\mu=\frac{1}{d\sqrt{t}}$ in the GauSGE and
$\mu=\frac{1}{\sqrt{d t}}$ in the GooSGE.
Meanwhile, we fix $\lambda_1 = \lambda_2 =10^{-5}$, and use the same initial solution $x_0$
from the standard normal distribution in each experiment.
For each dataset, we use half of the samples as training data,
and the rest as testing data.

\begin{table}
  \centering
  \caption{Real data for black-box binary classification } \label{tab:2}
  \begin{tabular}{c|c|c|c}
  \hline
  datasets & $\#samples$ & $\#features$ & $\#classes$ \\ \hline
  \emph{20news} & 16,242 &  100 & 2 \\
  \emph{a9a}   & 32,561 & 123 & 2 \\
  \emph{w8a} & 64,700 & 300 & 2 \\
  \emph{covtype.binary} & 581,012 & 54 & 2\\
  \hline
  \end{tabular}
\end{table}

\begin{figure*}[htbp]
\centering
\subfigure[20news]{\includegraphics[width=0.23\textwidth]{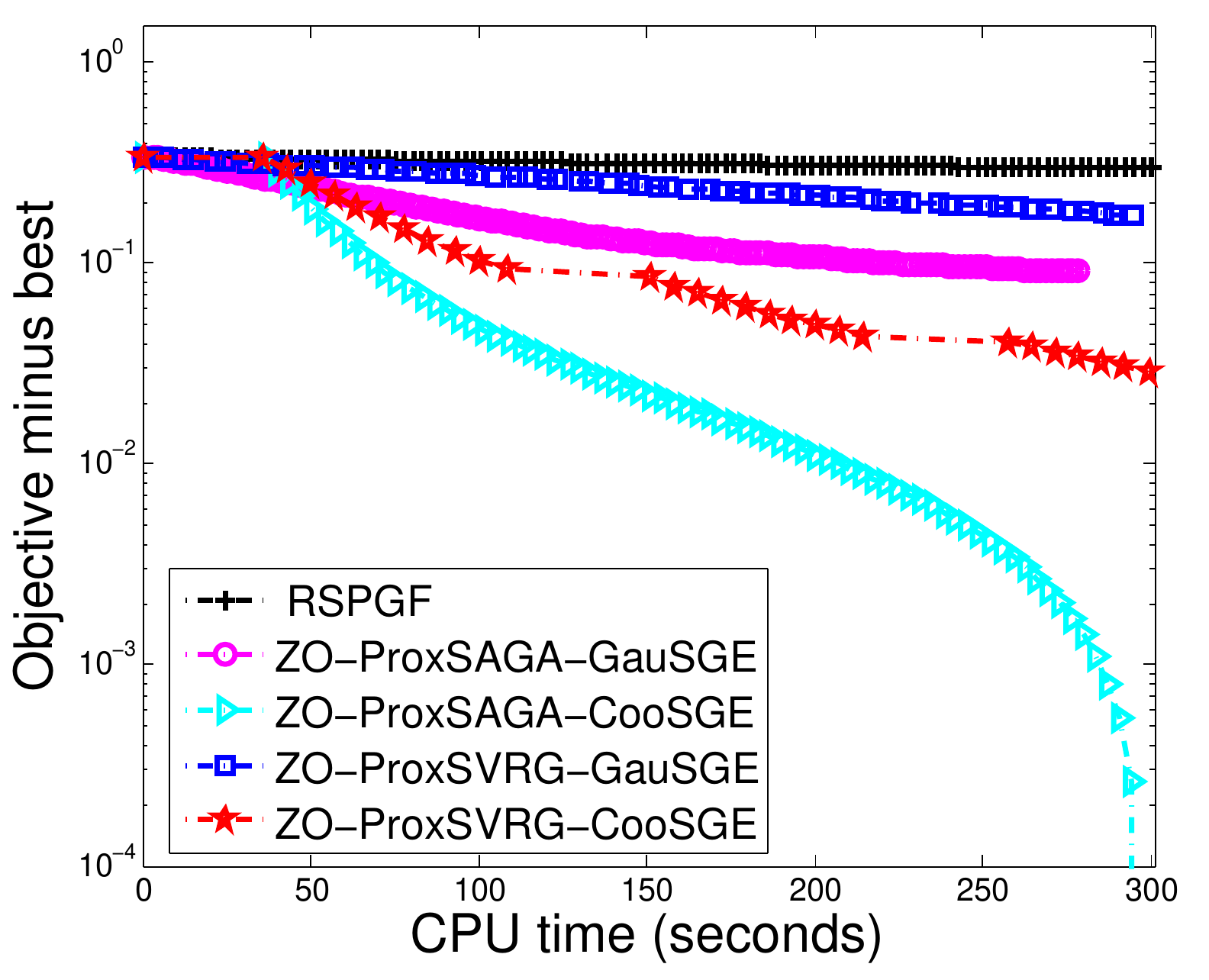}}
\subfigure[a9a]{\includegraphics[width=0.23\textwidth]{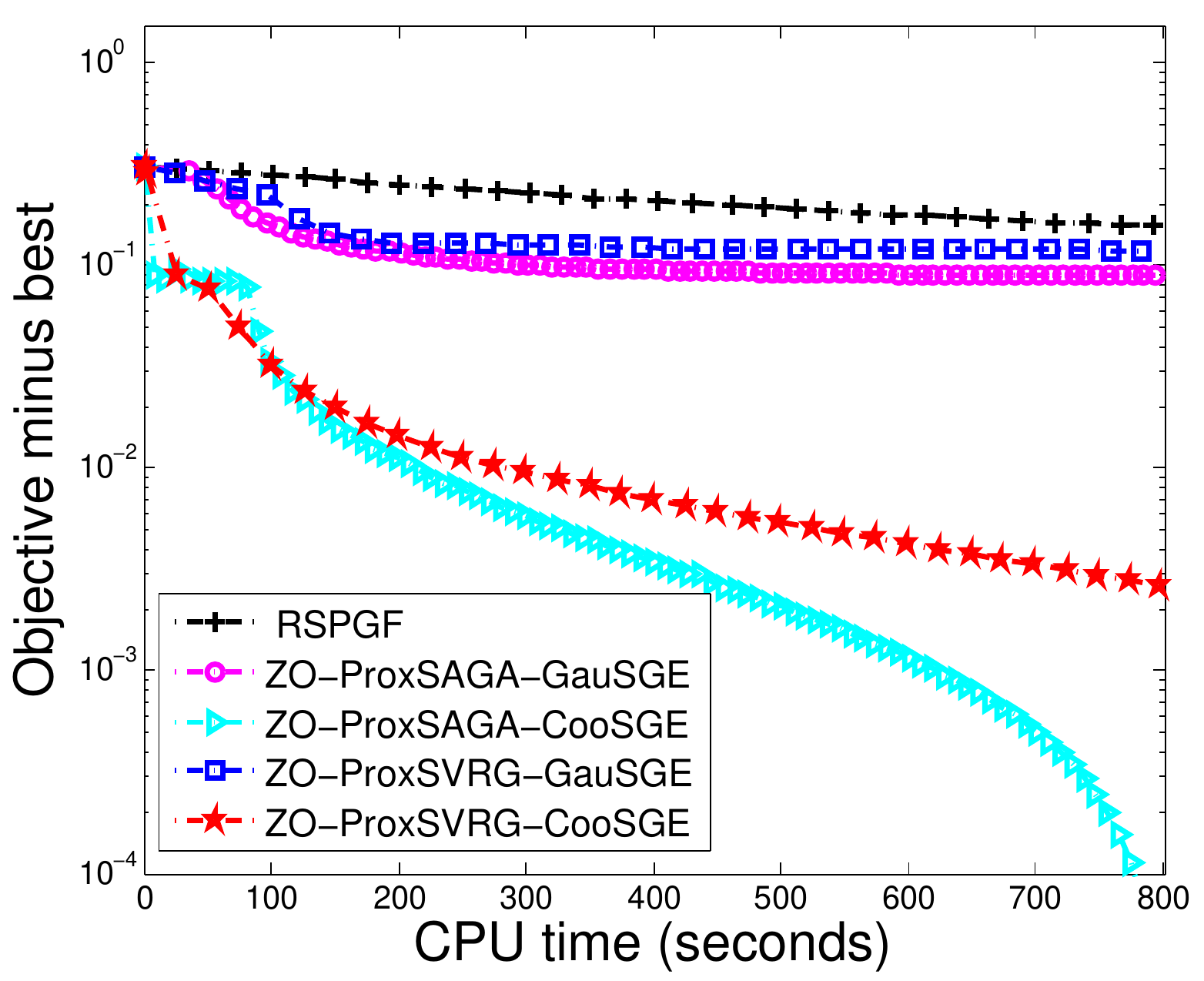}}
\subfigure[w8a]{\includegraphics[width=0.23\textwidth]{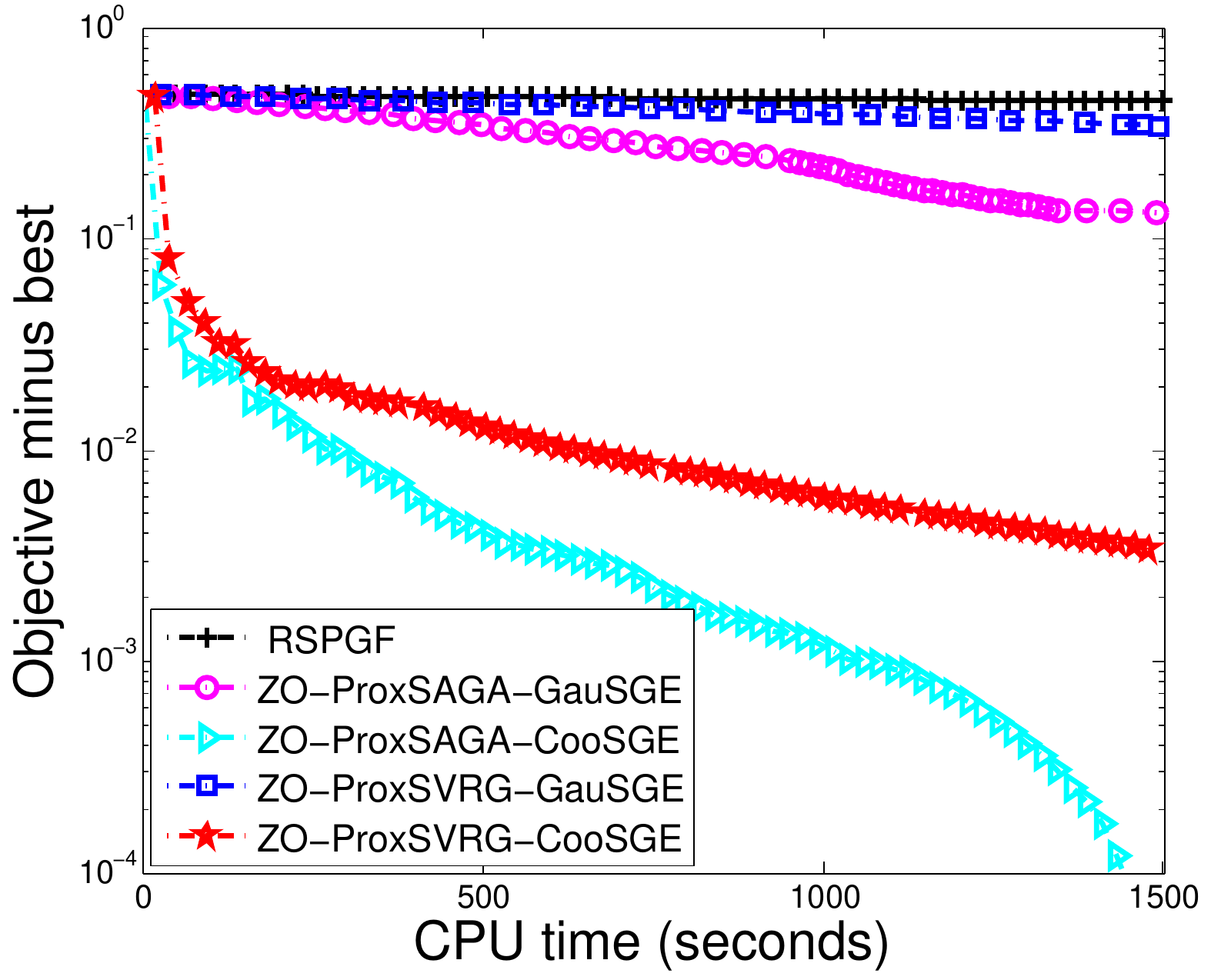}}
\subfigure[covtype.binary]{\includegraphics[width=0.23\textwidth]{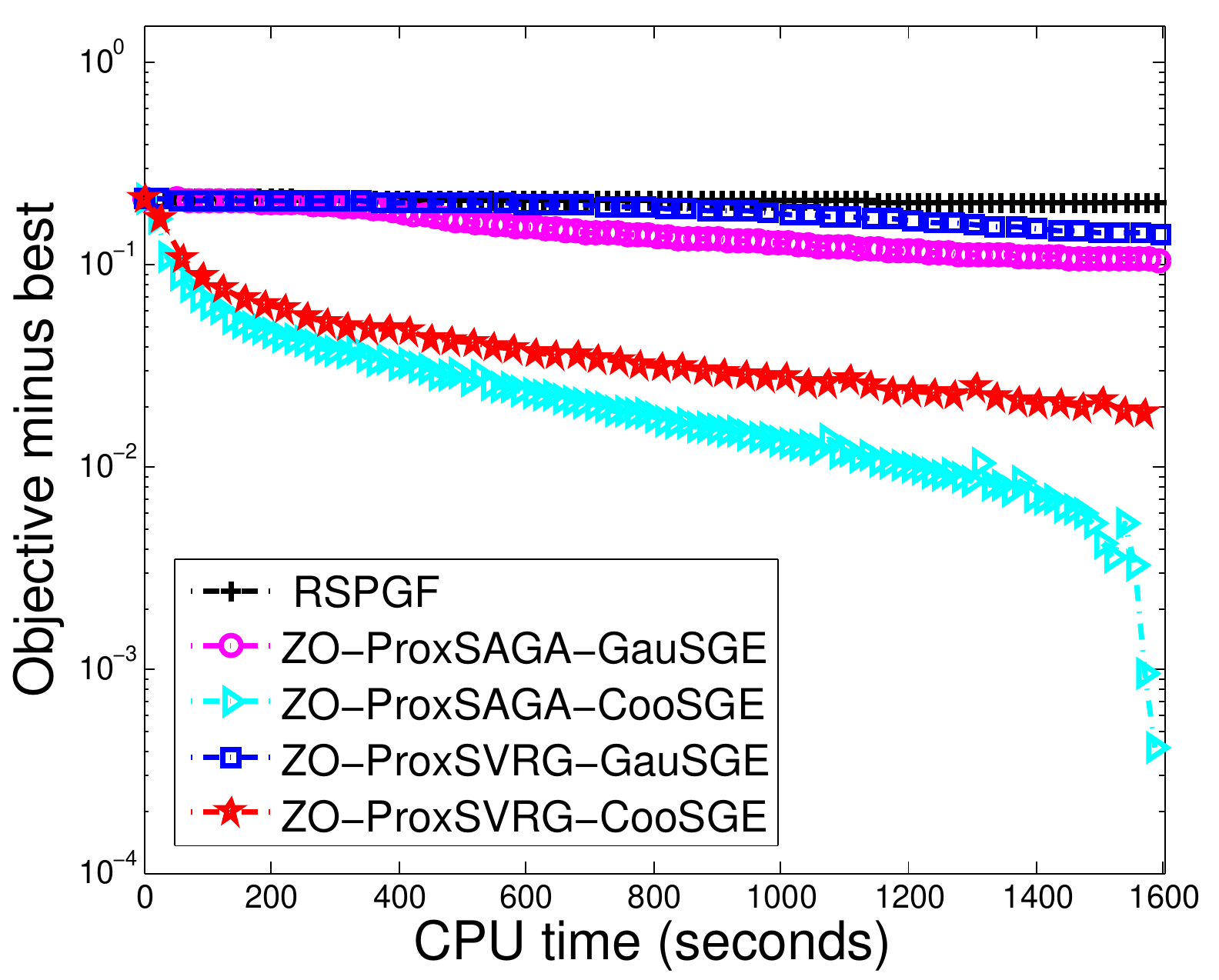}}
\caption{Objective value \emph{versus} CPU time on black-box binary classification.}
\label{fig:1}
\vspace{-1em}
\end{figure*}

\begin{figure*}[htbp]
\centering
\subfigure[20news]{\includegraphics[width=0.23\textwidth]{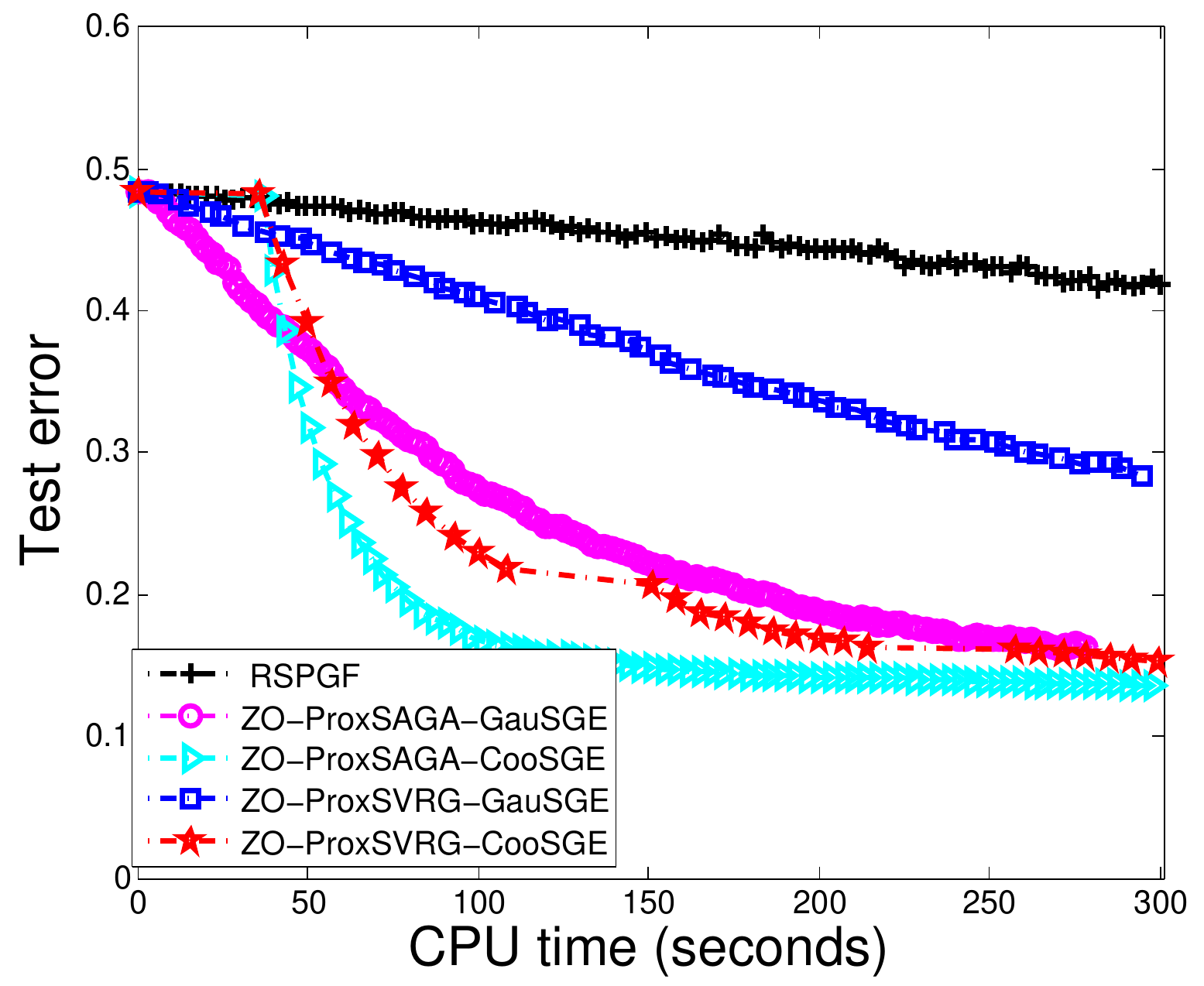}}
\subfigure[a9a]{\includegraphics[width=0.23\textwidth]{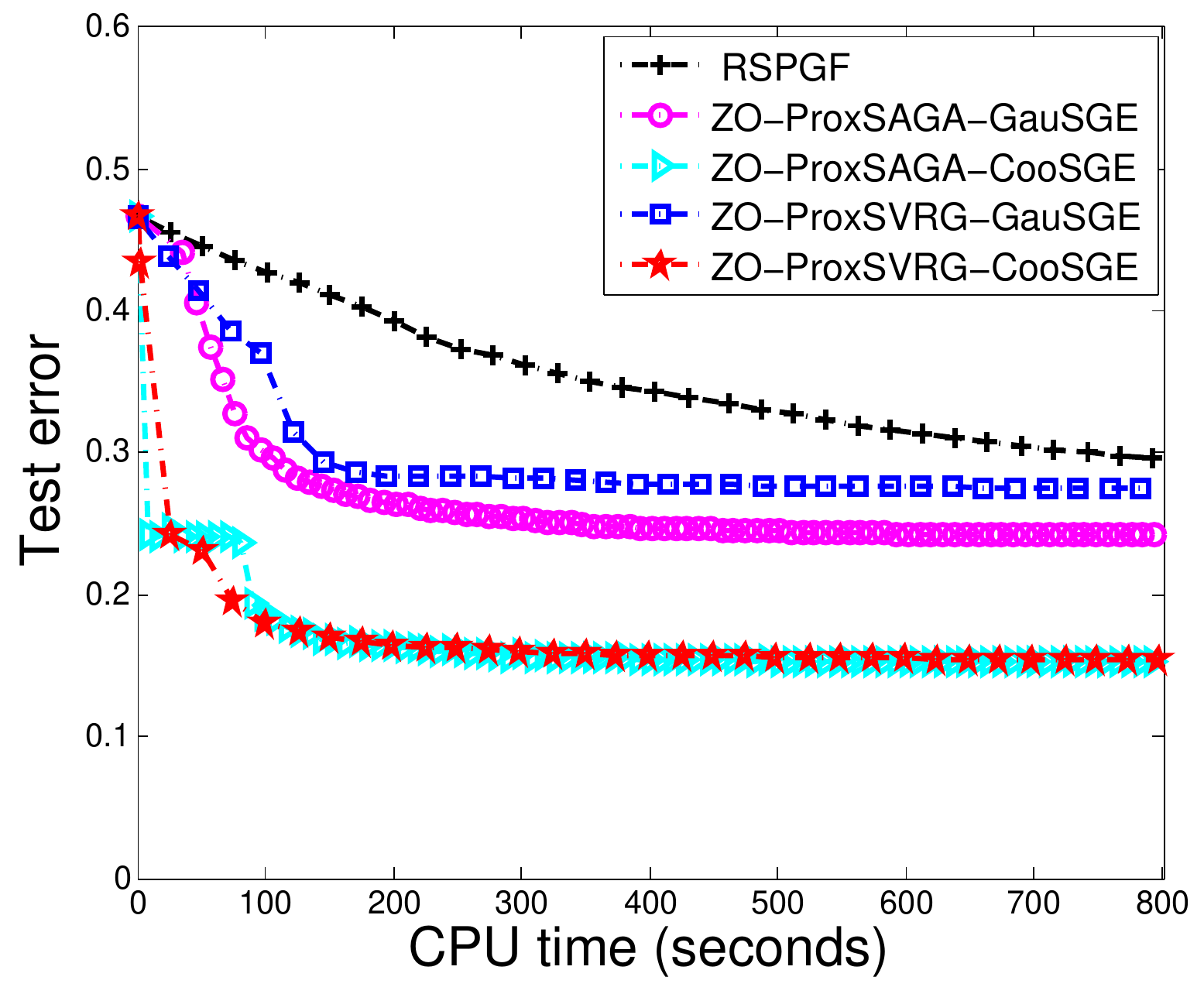}}
\subfigure[w8a]{\includegraphics[width=0.23\textwidth]{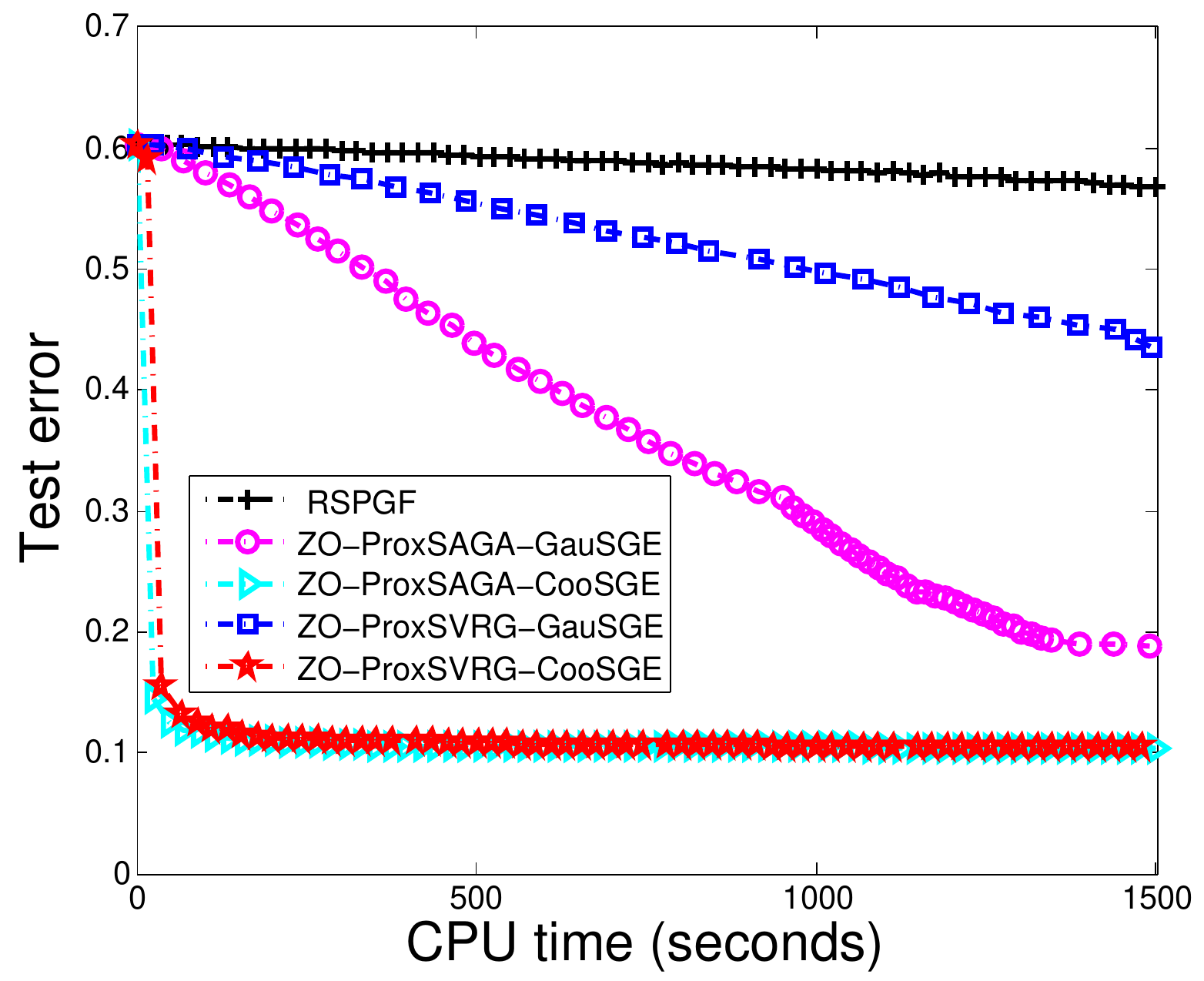}}
\subfigure[covtype.binary]{\includegraphics[width=0.23\textwidth]{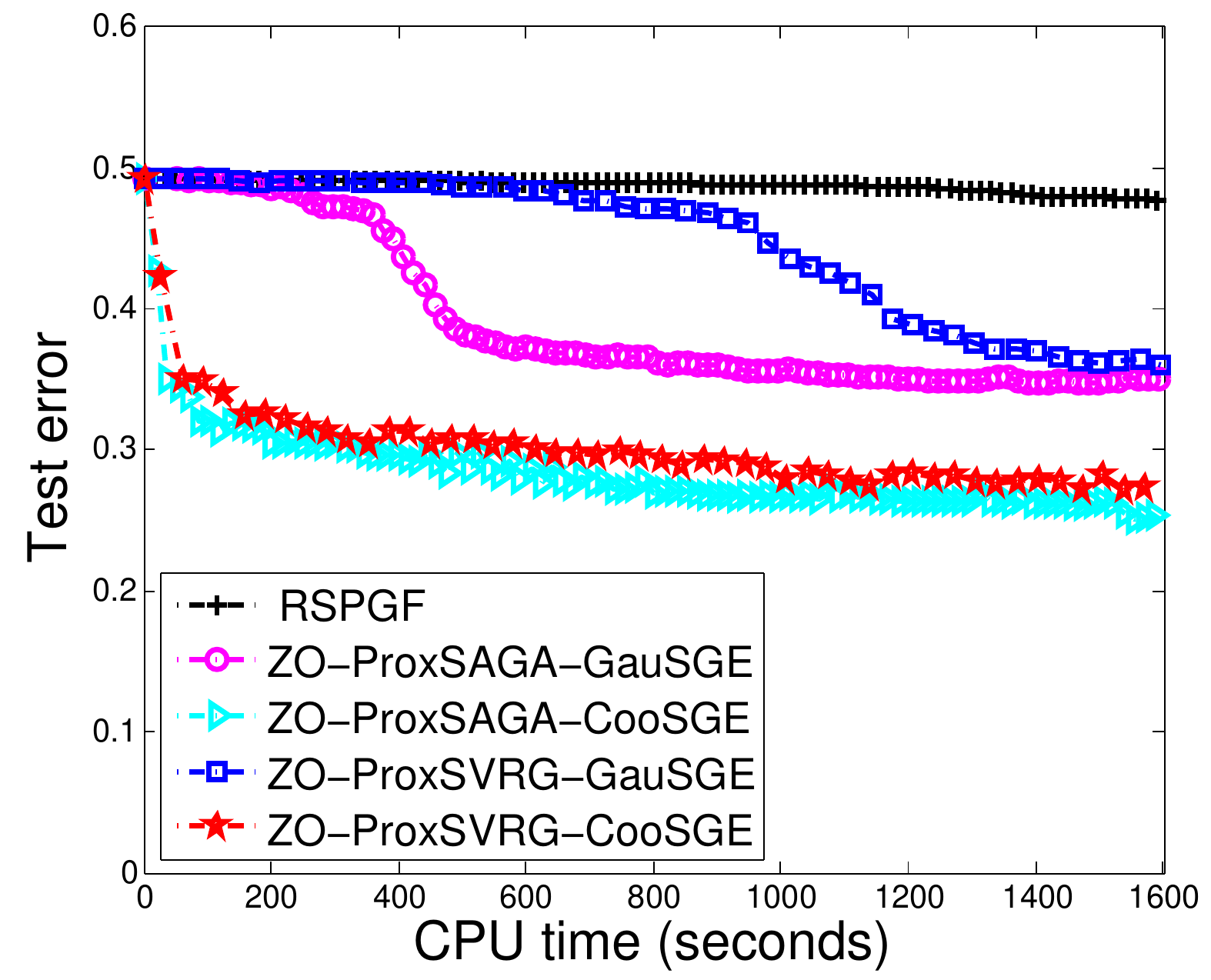}}
\caption{Test loss \emph{versus} CPU time on black-box binary classification.}
\label{fig:2}
\vspace{-2em}
\end{figure*}

\subsubsection{Experimental Results}
Figures \ref{fig:1} and \ref{fig:2} show that both objective values
and test losses of the proposed methods faster decrease than the
RSPGF method, as the time increases. In particular, both the
ZO-ProxSVRG and ZO-ProxSAGA using the CooSGE show the better
performances than the counterparts using the GauSGE. From these
results, we find that the CooSGE shows the better performances than
the CauSGE in estimating gradients. Moreover, these results also
demonstrate that both the ZO-ProxSVRG and ZO-ProxSAGA using the
CooSGE have a relatively faster convergence rate than the
counterparts using the GauSGE. Since the ZO-ProxSAGA has less
function query complexity than the ZO-ProxSVRG, it shows the better
performances than the ZO-ProxSVRG. For example, the
ZO-ProxSVRG-CooSGE needs $O(ndS+bdT)$ function queries, while ZO-SAGA-CooSGE needs $O(bdT)$ function queries.

\subsection{Adversarial Attacks on Black-Box DNNs }
In this experiment, we apply our methods to generate adversarial examples to attack a pre-trained neural network model.
Following \citep{chen2017zoo,liu2018zeroth}, the parameters of given model are hidden from us
and only its outputs are accessible.
In this case, we can not compute the gradients by using back-propagation algorithm.
Thus, we use the zeroth-order algorithms to find an universal adversarial perturbation $x \in \mathbb{R}^d$
that could fool the samples $\{a_i\in \mathbb{R}^d , \ l_i \in \mathbb{N} \}_{i=1}^n$,
which can be specified as the following elastic-net attacks to black-box DNNs problem:
\begin{align} \label{eq:A27}
\min_{x \in \mathbb{R}^d} & \frac{1}{n}\sum\limits_{i=1}^n \max\big\{F_{l_i}(a_i+x) - \max\limits_{j\neq l_i} F_j(a_{i}+x), 0 \big\} \nonumber \\
&+ \lambda_1\left\|x \right\|_1 + \lambda_2 \left\|x \right\|_2^2,
\end{align}
where $\lambda_1$ and $\lambda_2$ are nonnegative parameters to balance attack success rate, distortion and sparsity.
Here $F(a) = [F_1(a),\cdots,F_K(a)] \in [0,1]^{K}$ represents the final layer output of neural network,
which is the probabilities of $K$ classes.

Following \citep{liu2018zeroth},  we use a pre-trained DNN\footnote{https://github.com/carlini/nn$\_$robust$\_$attacks.}
on the MNIST dataset as the target black-box model, which achieves 99.4$\%$ test accuracy.
In the experiment, we select $n=10$ examples from the same class, and set the batch size $b=5$ and
a constant step size $\eta = 1/d$ for the zeroth-order algorithms, where $d=28\times 28$.
In addition, we set $\lambda_1=10^{-3}$ and $\lambda_2=1$ in the experiment.

Figure \ref{fig:3} shows that both objective values and black-box attack losses (\emph{i.e.} the first part of the problem \eqref{eq:A27}) of
the proposed algorithms faster decrease than the RSPGF method, as the number of iteration increases.
Here, we add the ZO-ProxSGD-CooSGE method for comparison, which is obtained by combining the ZO-ProxSGD method with
the CooSGE.
Interestingly, the ZO-ProxSGD-CooSGE shows better performance than both the ZO-ProxSVRG-GauSGE and ZO-ProxSAGA-GauSGE,
which further demonstrates that the CooSGE can have better performance than the CauSGE in estimating gradient.
Although having a relatively good performance in generating the adversarial samples, the ZO-ProxSGD
still shows worse performance than both the ZO-ProxSVRG-CooSGE and ZO-ProxSAGA-CooSGE,
due to not using the VR technique.

\begin{figure}[htbp]
\centering
\hspace*{-16pt}\subfigure[\emph{Objective function value}]{\includegraphics[width=0.24\textwidth]{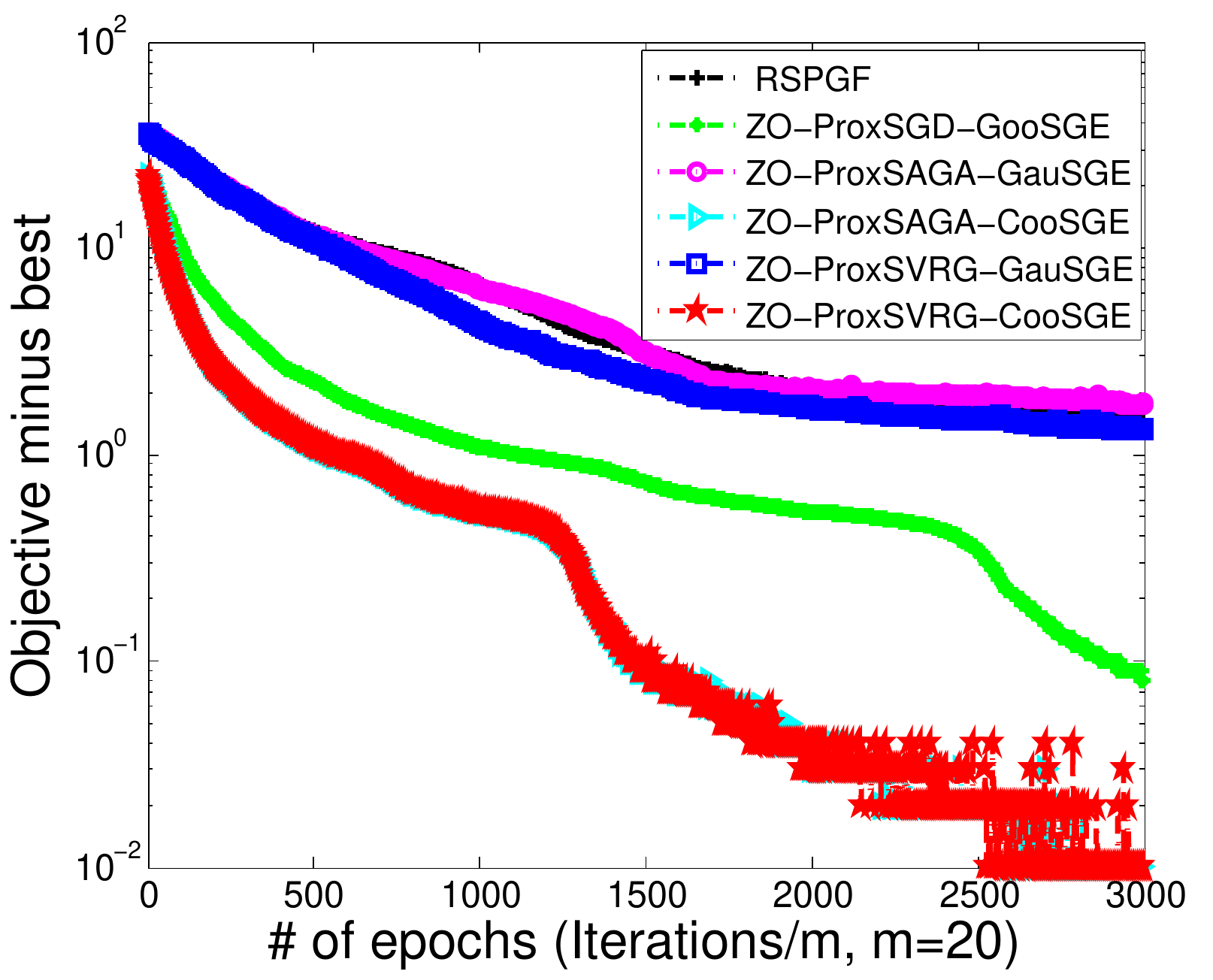}}
\subfigure[\emph{Black-box attack loss}]{\includegraphics[width=0.24\textwidth]{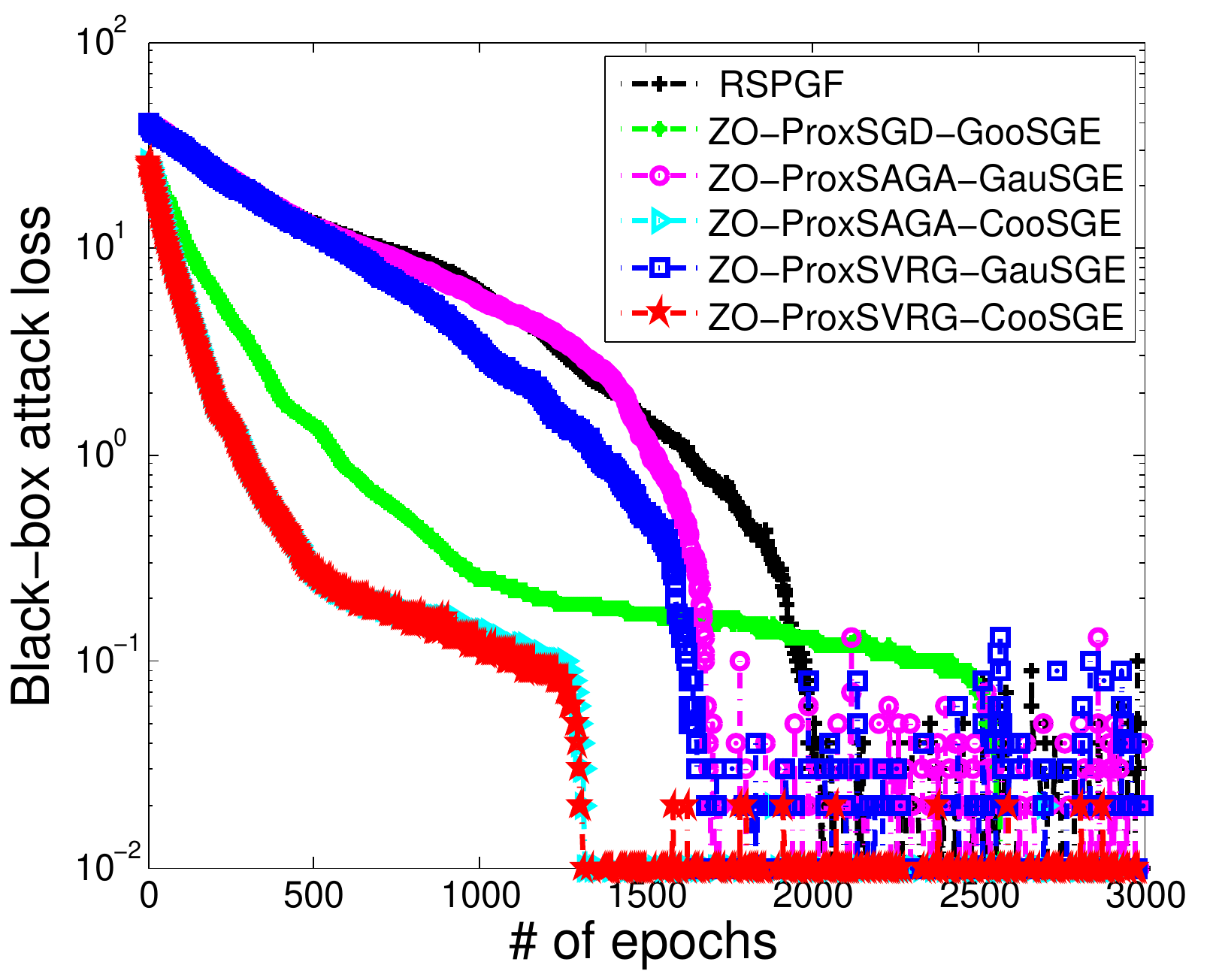}}
\caption{Objective value and attack loss on generating adversarial samples from black-box DNNs. }
\label{fig:3}
\vspace{-2em}
\end{figure}

\section{Conclusions}

In this paper, we proposed a class of faster gradient-free proximal stochastic methods
based on the zeroth-order gradient estimators, \emph{i.e.}, the GauSGE and the CooSGE,
which only use the objective function values in the optimization.
Moreover, we provided the theoretical analysis on the convergence properties of the proposed algorithms (ZO-ProxSVRG and ZO-ProxSAGA)
based on the CooSGE and the GauSGE, respectively.
In particular, both the ZO-ProxSVRG and ZO-ProxSAGA using the CooSGE have relatively faster convergence rates than
the counterparts using the GauSGE,
since the CooSGE has better performance than the CauSGE in estimating gradients.

\section{Acknowledgments}

F. Huang and S. Chen were partially supported by the Natural Science Foundation of China (NSFC) under Grant No. 61806093 and No. 61682281,
and the Key Program of NSFC under Grant No. 61732006, and Jiangsu Postdoctoral Research Grant Program No. 2018K004A.
F. Huang, Z. Huo, H. Huang were partially supported by U.S. NSF IIS 1836945,  IIS 1836938,  DBI 1836866,  IIS 1845666,  IIS 1852606,  IIS 1838627, IIS 1837956.

\bibliographystyle{aaai}
\bibliography{reference}

\begin{onecolumn}

\begin{appendices}

\section{Supplementary Materials for ``Faster Gradient-Free Proximal Stochastic Methods for Nonconvex Nonsmooth Optimization"}

In this section, we provide the detailed proofs of the above lemmas and theorems.
First, we give some useful properties of the CooSGE and the GauSGE, respectively.

\begin{lemma} \citep{liu2018zeroth}  \label{lem:5}
 Assume that the function $f(x)$ is $L$-smooth. Let $\hat{\nabla} f(x)$ denote the estimated gradient defined by the \textbf{CooSGE}.
 Define $f_{\mu_{j}} = \mathbb{E}_{u\sim
 U[-\mu_{j},\mu_{j}]}f(x+ue_{j})$, where
 $U[-\mu_{j},\mu_{j}]$ denotes the uniform distribution at the
 interval $[-\mu_{j},\mu_{j}]$. Then we have
\begin{itemize}
\item[1)] $f_{\mu_{j}}$ is $L$-smooth, and
\begin{align} \label{eq:29}
 \hat{\nabla} f(x) = \sum_{j=1}^d \frac{\partial f_{\mu_j}(x)}{\partial
 x_j} e_j,
\end{align}
where $\partial
f/\partial x_j$ denotes the partial derivative with respect to the
$j$th coordinate.

\item[2)] For $j\in [d]$,
\begin{align}
 & |f_{\mu_j}(x)-f(x)| \leq \frac{L\mu_j^2}{2}, \\
 & |\frac{\partial f_{\mu_j}(x)}{\partial x_j}| \leq
 \frac{L\mu_j^2}{2}.
\end{align}

\item[3)] If $\mu=\mu_j$ for $j\in [d]$, then
\begin{align} \label{eq:32}
 \|\hat{\nabla} f(x) - \nabla f(x)\|^2_2 \leq \frac{L^2d^2\mu^2}{4}.
\end{align}

\end{itemize}
\end{lemma}

\begin{lemma} \label{lem:6}
Assume that the function $f(x)$ is $L$-smooth. Let $\hat{\nabla} f(x)$ denote the estimated gradient defined by the \textbf{GauSGE}.
Define $f_{\mu}(x) = \mathbb{E}_{u\sim \mathcal{N}(0,I)}[f(x+\mu u)]$. Then we have
\begin{itemize}
\item[1)] For any $x\in \mathbb{R}^d$,  $\nabla f_{\mu}(x) = \mathbb{E}_u[\hat{\nabla} f(x)]$.
\item[2)] For any $x\in \mathbb{R}^d$,
\begin{align}
 & |f_{\mu}(x)-f(x)| \leq \frac{Ld\mu^2}{2}, \nonumber \\
 & |\nabla f_{\mu}(x)- \nabla f(x)| \leq \frac{L\mu(d+3)^{\frac{3}{2}}}{2}, \nonumber \\
 & \mathbb{E}_u\|\hat{\nabla} f(x) \|^2 \leq 2(d+4)\|\nabla f(x)\|^2 + \frac{\mu^2L^2(d+6)^3}{2}. \label{eq:33}
\end{align}
\item[3)] For any $x\in \mathbb{R}^d$,
\begin{align}  \label{eq:34}
 \mathbb{E}_u \|\hat{\nabla} f(x) - \nabla f(x)\|^2 \leq 2(2d+9)\|\nabla f(x)\|^2 + \mu^2L^2(d+6)^3.
\end{align}
\end{itemize}
\end{lemma}

\begin{proof}
The first and second parts of the above results can be obtain from Lemma 5 in \citep{ghadimi2016mini}.
Using the inequality \eqref{eq:33}, we have
\begin{align}
 \|\hat{\nabla} f(x)-\nabla f(x)\|^2 & \leq 2\|\hat{\nabla} f(x)\|^2 + 2 \|\nabla f(x)\|^2 \nonumber \\
 & \leq 2(2d+9)\|\nabla f(x)\|^2 + \mu^2L^2(d+6)^3, \nonumber
\end{align}
where the first inequality holds by the Cauchy-Schwarz and Young's  inequality.
\end{proof}

\textbf{Notations:} To make the paper easier to follow, we give
the following notations:
\begin{itemize}
\item $\|\cdot\|$ denotes the vector $\ell_2$ norm and the matrix spectral norm, respectively.
\item $\mu$ denotes the smooth parameter of the gradient estimators (\emph{i.e.}, the CooSGE and GauSGE ).
\item $\eta$ denotes the step size of updating variable $x$.
\item $L$ denotes the Lipschitz constant of $\nabla f(x)$.
\item $b$ denotes the mini-batch size of stochastic gradient.
\item $T$, $m$ and $S$ are the total number of iterations, the number of iterations in the inner loop, and the number
 of iterations in the outer loop, respectively.
\item For notational simplicity, $\mathbb{E}$ denotes $\mathbb{E}_{\mathcal{I}_t,u}$.
\end{itemize}

\subsection{Convergence Analysis of ZO-ProxSVRG-CooSGE}
In this section, we give the convergence analysis of the ZO-ProxSVRG-CooSGE.
First, we give an useful lemma about the upper bound of the variance of estimated gradient.

\begin{lemma} \label{lem:7}
In Algorithm \ref{alg:1} using the CooSGE, given the estimated gradient
$\hat{v}^s_t = \hat{\nabla} f_{\mathcal{I}_t}(x_{t}^{s})-\hat{\nabla} f_{\mathcal{I}_t}(\tilde{x}^s)+\hat{\nabla} f(\tilde{x}^s)$,
then the following inequality holds
\begin{align}
 \mathbb{E}\|\hat{v}^s_t-\nabla f(x^s_t)\|^2 \leq \frac{2\delta_n L^2
 d}{b}\mathbb{E} \|x^s_t-\tilde{x}^s\|^2_2 + \frac{L^2 d^2\mu^2}{2}.
\end{align}
\end{lemma}

\begin{proof}
Since
\begin{align} \label{eq:36}
\mathbb{E}_{\mathcal{I}_t}[\hat{\nabla} f_{\mathcal{I}_t}(x_{t}^{s})-\hat{\nabla} f_{\mathcal{I}_t}(\tilde{x}^s)] =
\hat{\nabla} f(x_{t}^{s})-\hat{\nabla} f(\tilde{x}^s),
\end{align}
we have
\begin{align} \label{eq:37}
\mathbb{E}\|\hat{v}^s_t-\nabla f(x^s_t)\|^2 &= \mathbb{E}\|\hat{\nabla} f_{\mathcal{I}_t}(x_{t}^{s})-\hat{\nabla} f_{\mathcal{I}_t}(\tilde{x}^s)
+\hat{\nabla} f(\tilde{x}^s)-\nabla f(x^s_t)\|^2 \nonumber \\
& = \mathbb{E}\|\hat{\nabla} f_{\mathcal{I}_t}(x_{t}^{s})-\hat{\nabla} f_{\mathcal{I}_t}(\tilde{x}^s)+\hat{\nabla} f(\tilde{x}^s) -
\hat{\nabla} f(x_{t}^{s}) + \hat{\nabla} f(x_{t}^{s}) -\nabla f(x^s_t)\|^2 \nonumber \\
& = \mathbb{E}\| \hat{\nabla} f_{\mathcal{I}_t}(x_{t}^{s})-\hat{\nabla} f_{\mathcal{I}_t}(\tilde{x}^s) -
\mathbb{E}_{\mathcal{I}_t}(\hat{\nabla} f_{\mathcal{I}_t}(x_{t}^{s})-\hat{\nabla} f_{\mathcal{I}_t}(\tilde{x}^s)) + \hat{\nabla} f(x_{t}^{s})
-\nabla f(x^s_t)\|^2  \nonumber \\
& \leq 2 \mathbb{E}\| \hat{\nabla} f_{\mathcal{I}_t}(x_{t}^{s})-\hat{\nabla} f_{\mathcal{I}_t}(\tilde{x}^s) -
\mathbb{E}_{\mathcal{I}_t}(\hat{\nabla} f_{\mathcal{I}_t}(x_{t}^{s})-\hat{\nabla} f_{\mathcal{I}_t}(\tilde{x}^s))\|^2 + 2 \mathbb{E}\|\hat{\nabla} f(x_{t}^{s})
-\nabla f(x^s_t)\|^2 \nonumber \\
& \leq 2\mathbb{E}\| \hat{\nabla} f_{\mathcal{I}_t}(x_{t}^{s})-\hat{\nabla} f_{\mathcal{I}_t}(\tilde{x}^s) -
\mathbb{E}_{\mathcal{I}_t}(\hat{\nabla} f_{\mathcal{I}_t}(x_{t}^{s})-\hat{\nabla} f_{\mathcal{I}_t}(\tilde{x}^s))\|^2 + \frac{L^2d^2\mu^2}{2},
\end{align}
where the second inequality holds by Lemma \ref{lem:5}. By the
equality \eqref{eq:36}, we have
\begin{align} \label{eq:36a}
\sum_{i=1}^n \big( \hat{\nabla} f_{i}(x_{t}^{s})-\hat{\nabla} f_{i}(\tilde{x}^s)
- \mathbb{E}_{\mathcal{I}_t}(\hat{\nabla} f_{\mathcal{I}_t}(x_{t}^{s})-\hat{\nabla} f_{\mathcal{I}_t}(\tilde{x}^s)) \big)
& = n \big( \hat{\nabla} f(x_{t}^{s})-\hat{\nabla} f(\tilde{x}^s)\big) - n \big( \hat{\nabla} f(x_{t}^{s})-\hat{\nabla} f(\tilde{x}^s)\big) \nonumber \\ &= 0.
\end{align}
Based on \eqref{eq:36a}, we have
\begin{align} \label{eq:39}
& \mathbb{E}\| \hat{\nabla} f_{\mathcal{I}_t}(x_{t}^{s})-\hat{\nabla} f_{\mathcal{I}_t}(\tilde{x}^s) - \mathbb{E}_{\mathcal{I}_t}(\hat{\nabla} f_{\mathcal{I}_t}(x_{t}^{s})-\hat{\nabla} f_{\mathcal{I}_t}(\tilde{x}^s))\|^2 \nonumber \\
& \leq \frac{\delta_n}{bn}\sum_{i=1}^n \mathbb{E} \|\hat{\nabla} f_i(x^s_t) - \hat{\nabla} f_i(\tilde{x}^s) - (\hat{\nabla} f(x^s_t) - \hat{\nabla} f(\tilde{x}^s))\|^2 \nonumber \\
& = \frac{\delta_n}{bn}\sum_{i=1}^n \mathbb{E} \|\hat{\nabla} f_i(x^s_t) - \hat{\nabla} f_i(\tilde{x}^s)\|^2\ - \|\hat{\nabla} f(x^s_t) - \hat{\nabla} f(\tilde{x}^s)\|^2 \nonumber \\
& \leq  \frac{\delta_n}{bn}\sum_{i=1}^n \mathbb{E} \| \hat{\nabla} f_i(x^s_t) - \hat{\nabla} f_i(\tilde{x}^s) \|^2,
\end{align}
where the first inequality holds by Lemmas 4 and 5 in \citep{liu2018zeroth}, and
\begin{equation*}
\delta_n = \left\{
\begin{aligned}
& 1, \quad \mbox{if $\mathcal{I}_t$ contains i.i.d. samples with replacement} \\
& I(b<n), \mbox{if $\mathcal{I}_t$ contains i.i.d. samples without replacement},
\end{aligned} \right.
\end{equation*}
$I(b<n)=1$ if $b<n$ and $0$ otherwise.
Using \eqref{eq:29}, we have
\begin{align} \label{eq:40}
\mathbb{E}\| \hat{\nabla} f_i(x^s_t) - \hat{\nabla} f_i(\tilde{x}^s) \|^2 & = \mathbb{E}\| \sum_{j=1}^d  \frac{\partial f_{i,\mu_j}}{\partial x^s_{t,j} }e_j - \frac{\partial f_{i,\mu_j}}{\partial \tilde{x}^s_j}e_j \|^2 \nonumber \\
& \leq d \sum_{j=1}^d \mathbb{E} \| \frac{\partial f_{i,\mu_j}}{\partial x^s_{t,j}}
- \frac{\partial f_{i,\mu_j}}{\partial \tilde{x}^s_j} \|^2 \nonumber \\
& \leq L^2d\sum_{j=1}^d \mathbb{E}\|x^s_{t,j} - \tilde{x}^s_j\|^2 = L^2d\|x^s_t-\tilde{x}^s\|^2,
\end{align}
where the first inequality holds by the Jensen's inequality yielding $\|\frac{1}{n}\sum_{i=1}^n z_i\|^2 \leq \frac{1}{n}\sum_{i=1}^n \|z_i\|^2$, and
the second inequality holds due to that the function $f_{\mu_j}$ is $L$-smooth.
Finally, combining the inequalities \eqref{eq:37}, \eqref{eq:39} and \eqref{eq:40}, we have
the above result.

\end{proof}

Next, based on the above lemma, we study the convergence property of the ZO-ProxSVRG-CooSGE.

\begin{theorem}
 Assume the sequence $\{(x^s_t)_{t=1}^m\}_{s=1}^S$ generated from Algorithm \ref{alg:1} using the \textbf{CooSGE},
 and given a sequence $\{c_t\}_{t=1}^m$ as follows: for $s=1,2,\cdots,S$
  \begin{equation}
   c_t = \left\{
  \begin{aligned}
   & \frac{\delta_n L^2d\eta}{b} + c_{t+1}(1+\beta), \ 0 \leq t \leq m-1; \\
   & 0, \ t = m
 \end{aligned}
  \right.\end{equation}
 where $\beta>0$. Let $T=mS$, $\eta = \frac{\rho}{dL} \ (0<\rho<\frac{1}{2})$ and $b$ satisfies the following inequality:
 \begin{align}
 \frac{8\rho^2 m^2}{b} + \rho \leq 1,
 \end{align}
 then we have
 \begin{align}
  \mathbb{E} \|g_{\eta}(x^s_t)\|^2 \leq \frac{\mathbb{E} [F(x^1_0)-F(x_*)]}{T\gamma} + \frac{ L^2 d^2\mu^2\eta}{4\gamma},
 \end{align}
 where $\gamma=\frac{\eta}{2}-L\eta^2$ and $x^*$ is an optimal solution of the problem \eqref{eq:1}.
 Further let $b=[n^{\frac{2}{3}}]$, $m=[n^{\frac{1}{4}}]$, $\rho=\frac{1}{4}$ and $\mu=O(\frac{1}{\sqrt{dT}})$,
 we have
 \begin{align}
 \mathbb{E} \|g_{\eta}(x^s_t)\|^2 \leq \frac{ 16d L\mathbb{E} [F(x^1_0) - F(x_*)]}{T} + O(\frac{d}{T}).
 \end{align}
\end{theorem}

\begin{proof}
 We begin with defining an iteration by using the full true gradient:
 \begin{align}
  \bar{x}^s_{t+1} = \mbox{Prox}_{\eta\psi}\big(x^s_t - \eta\nabla f(x^s_t)\big).
 \end{align}
 Then applying Lemma 2 of \cite{Reddi2016Prox}, we have
 \begin{align} \label{eq:44}
  F(\bar{x}^s_{t+1}) \leq F(z) + (\frac{L}{2}-\frac{1}{2\eta})\|\bar{x}^s_{t+1}-x^s_t\|^2 + (\frac{L}{2}+\frac{1}{2\eta})\|z-x^s_t\|^2 - \frac{1}{2\eta}\|\bar{x}^s_{t+1}-z\|^2, \quad \forall z\in \mathbb{R}^d.
 \end{align}
 Since $x^s_{t+1}=\mbox{Prox}_{\eta\psi}\big(x^s_t - \eta \hat{v}^s_t \big)$, we have
 \begin{align} \label{eq:45}
  F(x^s_{t+1}) \leq F(z) + \langle x^s_{t+1}-z, \nabla f(x^s_t)-\hat{v}^s_t\rangle + (\frac{L}{2}-\frac{1}{2\eta})\|x^s_{t+1}-x^s_t\|^2 + (\frac{L}{2}+\frac{1}{2\eta})\|z-x^s_t\|^2 - \frac{1}{2\eta}\|x^s_{t+1}-z\|^2.
 \end{align}
 Setting $z=x^s_t$ in \eqref{eq:44} and $z=\bar{x}^s_{t+1}$ in \eqref{eq:45},
 then summing them together and taking the expectations, we have
 \begin{align} \label{eq:46}
  \mathbb{E}[F(x^s_{t+1})] \leq & \mathbb{E}\big[ F(x^s_{t}) + \underbrace{\langle x^s_{t+1}-\bar{x}^s_{t+1}, \nabla f(x^s_t)-\hat{v}^s_t\rangle}_{T_1} +  (\frac{L}{2}-\frac{1}{2\eta})\|x^s_{t+1}-x^s_t\|^2\nonumber \\
  &+ (L-\frac{1}{2\eta})\|\bar{x}^s_{t+1}-x^s_t\|^2 - \frac{1}{2\eta}\|x^s_{t+1}-\bar{x}^s_{t+1}\|^2 \big].
 \end{align}
 Next, we give an upper bound of the term $T_1$ as follows:
 \begin{align} \label{eq:47}
  T_1 & = \mathbb{E} \langle x^s_{t+1}-\bar{x}^s_{t+1}, \nabla f(x^s_t)-\hat{v}^s_t\rangle \nonumber \\
  & \leq \frac{1}{2\eta}\mathbb{E}\|x^s_{t+1}-\bar{x}^s_{t+1}\|^2 + \frac{\eta}{2}\mathbb{E}\|\nabla f(x^s_t)-\hat{v}^s_t\|^2 \nonumber \\
  & \leq \frac{1}{2\eta}\mathbb{E}\|x^s_{t+1}-\bar{x}^s_{t+1}\|^2 + \frac{\delta_n L^2d \eta}{b}\mathbb{E} \|x^s_t-\tilde{x}^s\|^2_2 + \frac{L^2 d^2\mu^2\eta}{4},
 \end{align}
 where the first inequality holds by Cauchy-Schwarz and Young's inequality and the second inequality holds by
 Lemma \ref{lem:7}.
 Combining \eqref{eq:46} with \eqref{eq:47}, we have
 \begin{align}
  \mathbb{E}[F(x^s_{t+1})] \leq & \mathbb{E}\big[ F(x^s_{t}) + \frac{\delta_n L^2d \eta}{b}\mathbb{E} \|x^s_t-\tilde{x}^s\|^2_2 + \frac{L^2 d^2\mu^2\eta}{4} +  (\frac{L}{2}-\frac{1}{2\eta})\|x^s_{t+1}-x^s_t\|^2\nonumber \\
  &+ (L-\frac{1}{2\eta})\|\bar{x}^s_{t+1}-x^s_t\|^2 \big].
 \end{align}

 Next, we define an useful \emph{Lyapunov} function as follows:
 \begin{align}
   R^s_t = \mathbb{E}\big[ F(x^s_t) + c_t \|x^s_t-\tilde{x}^s\|^2 \big],
 \end{align}
 where $\{c_t\}$ is a nonnegative sequence.
 Considering the upper bound of $\|x^s_{t+1}-\tilde{x}^s\|^2$, we have
 \begin{align}
  \|x^s_{t+1} - \tilde{x}^s\|^2 & = \|x^s_{t+1} - x^s_t + x^s_t - \tilde{x}^s\|^2 \nonumber \\
  & =  \|x^s_{t+1}-x^s_t\|^2 + 2(x^s_{t+1}-x^s_t)^T(x^s_t - \tilde{x}^s)+ \|x^s_t-\tilde{x}^s\|^2 \nonumber \\
  & \leq \|x^s_{t+1}-x^s_t\|^2 + 2\big( \frac{1}{2\beta}\|x^s_{t+1}-x^s_t)\|^2 + \frac{\beta}{2}\|x^s_t-\tilde{x}^s\|^2\big)+ \|x^s_t-\tilde{x}^s\|^2 \nonumber \\
  & = (1+\frac{1}{\beta})\|x^s_{t+1}-x^s_t\|^2 + (1+\beta)\|x^s_t-\tilde{x}^s\|^2,
 \end{align}
 where $\beta>0$. 
 Then we have
 \begin{align}
  R^s_{t+1} & = \mathbb{E}\big[ F(x^s_{t+1}) + c_{t+1}\|x^s_{t+1}-\tilde{x}^s\|^2\big] \nonumber \\
  & \leq \mathbb{E}\big[ F(x^s_{t+1}) + c_{t+1}(1+\frac{1}{\beta})\|x^s_{t+1}-x^s_t\|^2 + c_{t+1}(1+\beta)\|x^s_{t+1}-\tilde{x}^s\|^2\big] \nonumber \\
  & \leq \mathbb{E}\big[ F(x^s_t) +  \big(\frac{\delta_n L^2d\eta}{b} + c_{t+1}(1+\beta)\big)\|x^s_t-\tilde{x}^s\|^2
  + (L-\frac{1}{2\eta})\|\bar{x}^s_{t+1}-x^s_t\|^2  \nonumber \\
 & \quad + \big(\frac{L}{2}-\frac{1}{2\eta} + c_{t+1}(1+\frac{1}{\beta})\big)\|x^s_{t+1}-x^s_t\|^2 + \frac{L^2d^2\mu^2\eta}{4} \big]\nonumber \\
 & =  R^s_{t} + (L-\frac{1}{2\eta})\|\bar{x}^s_{t+1}-x^s_t\|^2 + \big(\frac{L}{2}-\frac{1}{2\eta} + c_{t+1}(1+\frac{1}{\beta})\big)\|x^s_{t+1}-x^s_t\|^2
  + \frac{L^2d^2\mu^2\eta}{4},
 \end{align}
where $c_t = \frac{\delta_n L^2d\eta}{b} + c_{t+1}(1+\beta)$.
Let $c_m = 0$, $\beta=\frac{1}{m}$ and $\eta=\frac{\rho}{dL} (0 < \rho < \frac{1}{2})$ , recursing on $t$, we have
\begin{align}
 c_t = \frac{\delta_n L^2d\eta}{b}\frac{(1+\beta)^{m-t}-1}{\beta} & = \frac{\delta_n L\rho m}{b}\big((1+\frac{1}{m})^{m-t}-1\big) \nonumber \\
 & \leq \frac{\delta_n L \rho m}{b}(e-1) \leq \frac{2 L \rho m}{b},
\end{align}
where the first inequality holds by $(1+\frac{1}{m})^m$ is an increasing function and $\lim_{m\rightarrow \infty}(1+\frac{1}{m})^m=e$; 
The second inequality holds by $0 \leq \delta_n \leq 1$ and $e-1\leq 2$.
It follows that
\begin{align}
 \frac{L}{2} + c_{t+1}(1+\frac{1}{\beta}) &\leq \frac{L}{2} + \frac{2 L \rho m}{b}(1+m) \nonumber \\
 &\leq \frac{L}{2} + \frac{4 L \rho m^2}{b} = (\rho + \frac{8 \rho^2 m^2}{b})\frac{L}{2\rho} \leq \frac{L}{2\rho}\leq \frac{1}{2\eta},
\end{align}
where the last inequality holds by $\rho + \frac{8\rho^2 m^2}{b} \leq 1$. Thus, we have $\frac{L}{2}-\frac{1}{2\eta} + c_{t+1}(1+\frac{1}{\beta}) \geq 0$. 
Then, we obtain
\begin{align} \label{eq:A59}
 R^s_{t+1} \leq R^s_{t} + (L-\frac{1}{2\eta})\|\bar{x}^s_{t+1}-x^s_t\|^2 + \frac{L^2d^2\mu^2\eta}{4}.
\end{align}

Telescoping inequality \eqref{eq:A59} over $t$ from $0$ to $m-1$, since $x^s_0=x^{s-1}_m=\tilde{x}^{s-1}$ and $x^{s}_m=\tilde{x}^{s}$, we have
\begin{align} \label{eq:A60}
 \frac{1}{m}\sum_{t=1}^m \mathbb{E}\|g_{\eta}(x^s_t)\|^2 \leq \frac{\mathbb{E}[F(\tilde{x}^{s-1})-F(\tilde{x}^s)]}{m\gamma} + \frac{L^2 d^2\mu^2\eta}{4\gamma},
\end{align}
where $\gamma = \frac{\eta}{2}-\eta^2L$, and
\begin{align}
 g_{\eta}(x^s_t) = \frac{1}{\eta}\big[ x^s_t - \mbox{Prox}_{\eta\psi}(x^s_t - \eta \nabla f(x^s_t))\big]=\frac{1}{\eta}(x^s_t-\bar{x}^s_{t+1}).
\end{align}
Summing the inequality \eqref{eq:A60} over $s$ from $1$ to $S$, we have
\begin{align} \label{eq:51}
 \min_{t,s} \mathbb{E}\|g_{\eta}(x^s_t)\|^2 \leq \frac{1}{T}\sum_{s=1}^S \sum_{t=1}^m \mathbb{E}\|g_{\eta}(x^s_t)\|^2
 &\leq \frac{\mathbb{E}[F(\tilde{x}^0)-F(\tilde{x}^S)]}{T\gamma} + \frac{L^2 d^2\mu^2\eta}{4\gamma} \nonumber \\
 & \leq \frac{\mathbb{E}[F(\tilde{x}^0)-F(x_*)]}{T\gamma} + \frac{L^2 d^2\mu^2\eta}{4\gamma},
\end{align}
where $x_*$ is an optimal solution of \eqref{eq:1}.

Given $m=[n^{\frac{1}{3}}]$, $b=[n^{\frac{2}{3}}]$ and $\rho = \frac{1}{4}$,
it is easy verified that $\rho + \frac{8\rho^2 m^2}{b} = \frac{3}{4} < 1$. Using $d \geq 1$, we have
$\gamma=\frac{\eta}{2}-L\eta^2=\frac{1}{8dL}-\frac{1}{16d^2L}\leq \frac{1}{8dL}-\frac{1}{16dL}=\frac{1}{16dL}$,
we can obtain the above results.

\end{proof}

\subsection{Convergence Analysis of ZO-ProxSVRG-GauSGE}
In this section, we give the convergence analysis of the ZO-ProxSVRG-GauSGE.
First, we give an useful lemma about the upper bound of the variance of estimated gradient.

\begin{lemma} \label{lem:8}
In Algorithm \ref{alg:1} using GauSGE, given the estimated gradient
$\hat{v}^s_t = \hat{\nabla} f_{\mathcal{I}_t}(x_{t}^{s})-\hat{\nabla} f_{\mathcal{I}_t}(\tilde{x}^s)+\hat{\nabla} f(\tilde{x}^s)$,
then the following inequality holds
\begin{align}
 \mathbb{E}\|\hat{v}^s_t-\nabla f(x^s_t)\|^2 \leq &  \frac{6\delta_nL^2}{b}\mathbb{E}\|x^s_t-\tilde{x}^s\|^2
    + (2+\frac{12\delta_n}{b})L^2\mu^2(d+6)^3 \nonumber \\
  & + (4 + \frac{24\delta_n}{b})(2d+9)\sigma^2.
\end{align}
\end{lemma}

\begin{proof}
Since
\begin{align} \label{eq:59}
\mathbb{E}_{\mathcal{I}_t}[\hat{\nabla} f_{\mathcal{I}_t}(x_{t}^{s})-\hat{\nabla} f_{\mathcal{I}_t}(\tilde{x}^s)] =
\hat{\nabla} f(x_{t}^{s})-\hat{\nabla} f(\tilde{x}^s),
\end{align}
we have
\begin{align} \label{eq:60}
\mathbb{E}\|\hat{v}^s_t-\nabla f(x^s_t)\|^2 &= \|\hat{\nabla} f_{\mathcal{I}_t}(x_{t}^{s})-\hat{\nabla} f_{\mathcal{I}_t}(\tilde{x}^s)+\hat{\nabla} f(\tilde{x}^s)-\nabla f(x^s_t)\|^2  \\
& = \mathbb{E}\|\hat{\nabla} f_{\mathcal{I}_t}(x_{t}^{s})-\hat{\nabla} f_{\mathcal{I}_t}(\tilde{x}^s)+\hat{\nabla} f(\tilde{x}^s) -
\hat{\nabla} f(x_{t}^{s}) + \hat{\nabla} f(x_{t}^{s}) -\nabla f(x^s_t)\|^2 \nonumber \\
& = \mathbb{E}\| \hat{\nabla} f_{\mathcal{I}_t}(x_{t}^{s})-\hat{\nabla} f_{\mathcal{I}_t}(\tilde{x}^s) -
\mathbb{E}_{\mathcal{I}_t}(\hat{\nabla} f_{\mathcal{I}_t}(x_{t}^{s})-\hat{\nabla} f_{\mathcal{I}_t}(\tilde{x}^s)) + \hat{\nabla} f(x_{t}^{s})
-\nabla f(x^s_t)\|^2  \nonumber \\
& \leq 2 \mathbb{E}\| \hat{\nabla} f_{\mathcal{I}_t}(x_{t}^{s})-\hat{\nabla} f_{\mathcal{I}_t}(\tilde{x}^s) -
\mathbb{E}_{\mathcal{I}_t}(\hat{\nabla} f_{\mathcal{I}_t}(x_{t}^{s})-\hat{\nabla} f_{\mathcal{I}_t}(\tilde{x}^s))\|^2 + 2\mathbb{E}\|\hat{\nabla} f(x_{t}^{s})
-\nabla f(x^s_t)\|^2 \nonumber \\
& \leq 2 \mathbb{E}\| \hat{\nabla} f_{\mathcal{I}_t}(x_{t}^{s})-\hat{\nabla} f_{\mathcal{I}_t}(\tilde{x}^s) -
\mathbb{E}_{\mathcal{I}_t}(\hat{\nabla} f_{\mathcal{I}_t}(x_{t}^{s})-\hat{\nabla} f_{\mathcal{I}_t}(\tilde{x}^s))\|^2
+ 4(2d+9)\|\nabla f(x^s_t)\|^2 + 2\mu^2L^2(d+6)^3, \nonumber \\
& \leq 2 \mathbb{E}\| \hat{\nabla} f_{\mathcal{I}_t}(x_{t}^{s})-\hat{\nabla} f_{\mathcal{I}_t}(\tilde{x}^s) -
\mathbb{E}_{\mathcal{I}_t}(\hat{\nabla} f_{\mathcal{I}_t}(x_{t}^{s})-\hat{\nabla} f_{\mathcal{I}_t}(\tilde{x}^s))\|^2
+ 4(2d+9)\sigma^2 + 2\mu^2L^2(d+6)^3, \nonumber
\end{align}
where the second inequality holds by Lemma \ref{lem:6} and the third inequality follows Assumption 2. By the
equality \eqref{eq:59}, we have
\begin{align}
\sum_{i=1}^n \big( \hat{\nabla} f_{i}(x_{t}^{s})-\hat{\nabla} f_{i}(\tilde{x}^s) - \mathbb{E}_{\mathcal{I}_t}(\hat{\nabla} f_{\mathcal{I}_t}(x_{t}^{s})-\hat{\nabla} f_{\mathcal{I}_t}(\tilde{x}^s)) \big)& = n \big( \hat{\nabla} f(x_{t}^{s})-\hat{\nabla} f(\tilde{x}^s)\big) - n \big( \hat{\nabla} f(x_{t}^{s})-\hat{\nabla} f(\tilde{x}^s)\big) \nonumber \\&= 0.
\end{align}
It follows that
\begin{align} \label{eq:62}
& \mathbb{E}\| \hat{\nabla} f_{\mathcal{I}_t}(x_{t}^{s})-\hat{\nabla} f_{\mathcal{I}_t}(\tilde{x}^s) - \mathbb{E}_{\mathcal{I}_t}(\hat{\nabla} f_{\mathcal{I}_t}(x_{t}^{s})-\hat{\nabla} f_{\mathcal{I}_t}(\tilde{x}^s))\|^2 \nonumber \\
& \leq \frac{\delta_n}{bn}\sum_{i=1}^n \mathbb{E} \|\hat{\nabla} f_i(x^s_t) - \hat{\nabla} f_i(\tilde{x}^s) - (\hat{\nabla} f(x^s_t) - \hat{\nabla} f(\tilde{x}^s))\|^2 \nonumber \\
& = \frac{\delta_n}{bn}\sum_{i=1}^n \mathbb{E} [\|\hat{\nabla} f_i(x^s_t) - \hat{\nabla} f_i(\tilde{x}^s)\|^2\ - \|(\hat{\nabla} f(x^s_t) - \hat{\nabla} f(\tilde{x}^s))\|^2 \nonumber \\
& \leq  \frac{\delta_n}{bn}\sum_{i=1}^n \mathbb{E} \| \hat{\nabla} f_i(x^s_t) - \hat{\nabla} f_i(\tilde{x}^s) \|^2,
\end{align}
where the first inequality holds by Lemmas 4 and 5 in \citep{liu2018zeroth}.
By \eqref{eq:59}, we have
\begin{align} \label{eq:64}
\mathbb{E}\| \hat{\nabla} f_i(x^s_t) - \hat{\nabla} f_i(\tilde{x}^s) \|^2 & = \mathbb{E}\| \hat{\nabla} f_i(x^s_t) - \nabla f_i (x^s_t) +  \nabla f_i (x^s_t)
- \nabla f_i(\tilde{x}^s) + \nabla f_i(\tilde{x}^s) - \hat{\nabla}f_i(\tilde{x}^s) \|^2 \nonumber \\
& \leq 3 \mathbb{E}\| \hat{\nabla} f_i(x^s_t) - \nabla f_i (x^s_t) \|^2 + 3 \|\nabla f_i (x^s_t)
- \nabla f_i(\tilde{x}^s)\|^2 + 3 \mathbb{E}\| \nabla f_i(\tilde{x}^s) - \hat{\nabla}f_i(\tilde{x}^s) \|^2 \nonumber \\
& \leq 6(2d+9)\big( \|\nabla f_i(x^s_t)\|^2 + \|\nabla f_i(\tilde{x}^s)\|^2 \big) + 3L^2\|x^s_t-\tilde{x}^s\|^2 + 6L^2\mu^2(d+6)^3 \nonumber \\
& \leq 12(2d+9)\sigma^2 + 3L^2\|x^s_t-\tilde{x}^s\|^2 + 6L^2\mu^2(d+6)^3,
\end{align}
where the first inequality holds by the Jensen's inequality,
the second inequality holds by the Lemma \ref{lem:6} and the third inequality follows Assumption 2.
Finally, combining the inequalities \eqref{eq:60}, \eqref{eq:62} and \eqref{eq:64}, we obtain
the above result.

\end{proof}

Next, based on the above lemma, we study the convergence property of the ZO-ProxSVRG-GauSGE.

\begin{theorem}
  Assume the sequence $\{(x^s_t)_{t=1}^m\}_{s=1}^S$ generated from Algorithm \ref{alg:1} using the GauSGE,
 and given a sequence $\{c_t\}_{t=1}^m$ as follows: for $s=1,2,\cdots,S$
  \begin{equation}
  c_t = \left\{
  \begin{aligned}
   & \frac{3\delta_n L^2\eta}{b} + c_{t+1}(1+\beta), \ 0 \leq t \leq m-1; \\
   & 0, \ t = m
 \end{aligned}
  \right.\end{equation}
 where $\beta>0$. Let $\eta = \frac{\rho}{L} \ (0<\rho<\frac{1}{2})$ and $b$ satisfies the following inequality:
 \begin{align}
  \frac{24\rho^2m^2}{b} + \rho \leq 1,
 \end{align}
 then we have
 \begin{align}
  \mathbb{E} \|g_{\eta}(x^s_t)\|^2 \leq \frac{\mathbb{E}[F(x^1_0)-F(x_*)]}{T\gamma}+ (1+\frac{6\delta_n}{b})(d+6)^3 \frac{L^2\mu^2\eta}{\gamma}  + (2 + \frac{12\delta_n}{b})(2d+9)\frac{\sigma^2\eta}{\gamma},
 \end{align}
 where $\gamma=\frac{\eta}{2}-\eta^2 L$ and $x^*$ is an optimal solution of the problem \eqref{eq:1}.
 Further let $b=[n^{\frac{2}{3}}]$, $m=[n^{\frac{1}{3}}]$, $\rho=\frac{1}{6}$ and $\mu=O(\frac{1}{d\sqrt{T}})$, we have
 \begin{align}
 \mathbb{E} \|g_{\eta}(x^s_t)\|^2 \leq  \frac{18L\mathbb{E}[F(x^1_0) - F(x_*)]}{T}
  + O(\frac{d}{T}) + O(d \sigma^2).
 \end{align}
\end{theorem}

\begin{proof}
This proof is the similar to the proof of Theorem \ref{th:1}.
We start by defining an iteration by using the full true gradient:
 \begin{align}
  \bar{x}^s_{t+1} = \mbox{Prox}_{\eta\psi}\big(x^s_t - \eta\nabla f(x^s_t)\big),
 \end{align}
 then applying Lemma 2 of \cite{Reddi2016Prox}, we have
 \begin{align} \label{eq:74}
  F(\bar{x}^s_{t+1}) \leq F(z) + (\frac{L}{2}-\frac{1}{2\eta})\|\bar{x}^s_{t+1}-x^s_t\|^2 + (\frac{L}{2}+\frac{1}{2\eta})\|z-x^s_t\|^2 - \frac{1}{2\eta}\|\bar{x}^s_{t+1}-z\|^2, \quad \forall z\in \mathbb{R}^d.
 \end{align}
Since $x^s_{t+1}=\mbox{Prox}_{\eta\psi}\big(x^s_t - \eta \hat{v}^s_t \big)$, we have
 \begin{align} \label{eq:75}
  F(x^s_{t+1}) \leq F(z) + \langle x^s_{t+1}-z, \nabla f(x^s_t)-\hat{v}^s_t\rangle + (\frac{L}{2}-\frac{1}{2\eta})\|x^s_{t+1}-x^s_t\|^2 + (\frac{L}{2}+\frac{1}{2\eta})\|z-x^s_t\|^2 - \frac{1}{2\eta}\|x^s_{t+1}-z\|^2.
 \end{align}
Setting $z=x^s_t$ in \eqref{eq:74} and $z=\bar{x}^s_{t+1}$ in \eqref{eq:75},
then summing them together and taking the expectations, we obtain
 \begin{align} \label{eq:76}
  \mathbb{E}[F(x^s_{t+1})] \leq & \mathbb{E}\big[ F(x^s_{t}) + \underbrace{\langle x^s_{t+1}-\bar{x}^s_{t+1}, \nabla f(x^s_t)-\hat{v}^s_t\rangle}_{T_2} +  (\frac{L}{2}-\frac{1}{2\eta})\|x^s_{t+1}-x^s_t\|^2\nonumber \\
  &+ (L-\frac{1}{2\eta})\|\bar{x}^s_{t+1}-x^s_t\|^2 - \frac{1}{2\eta}\|x^s_{t+1}-\bar{x}^s_{t+1}\|^2 \big].
 \end{align}
Next, we give an upper bound of the term $T_2$  as follows:
 \begin{align} \label{eq:77}
  T_2 & = \mathbb{E} \langle x^s_{t+1}-\bar{x}^s_{t+1}, \nabla f(x^s_t)-\hat{v}^s_t\rangle \nonumber \\
  & \leq \frac{1}{2\eta}\mathbb{E}\|x^s_{t+1}-\bar{x}^s_{t+1}\|^2 + \frac{\eta}{2}\mathbb{E}\|\nabla f(x^s_t)-\hat{v}^s_t\|^2 \nonumber \\
  & \leq \frac{1}{2\eta}\mathbb{E}\|x^s_{t+1}-\bar{x}^s_{t+1}\|^2 + \frac{3\delta_nL^2\eta}{b}\mathbb{E}\|x^s_t-\tilde{x}^s\|^2
    + (1+\frac{6\delta_n}{b})L^2(d+6)^3\mu^2\eta  + (2 + \frac{12\delta_n}{b})(2d+9)\sigma^2\eta,
 \end{align}
 where the first inequality holds by Cauchy-Schwarz and Young's inequality and the second inequality holds by
 Lemma \ref{lem:8}.
 Combining \eqref{eq:76} with \eqref{eq:77}, we have
 \begin{align}
  \mathbb{E}[F(x^s_{t+1})] \leq & \mathbb{E}\big[ F(x^s_{t}) + \frac{3\delta_nL^2\eta}{b}\mathbb{E}\|x^s_t-\tilde{x}^s\|^2
    + (1+\frac{6\delta_n}{b})L^2(d+6)^3\mu^2\eta  + (2 + \frac{12\delta_n}{b})(2d+9)\sigma^2\eta \nonumber \\
  & +  (\frac{L}{2}-\frac{1}{2\eta})\|x^s_{t+1}-x^s_t\|^2 + (L-\frac{1}{2\eta})\|\bar{x}^s_{t+1}-x^s_t\|^2 \big].
 \end{align}

 Next, we define an useful \emph{Lyapunov} function as follows:
 \begin{align}
   \Psi^s_t = \mathbb{E}\big[ F(x^s_t) + c_t \|x^s_t-\tilde{x}^s\|^2 \big],
 \end{align}
 where $\{c_t\}$ is a nonnegative sequence.
 Considering the upper bound of $\|x^s_{t+1}-\tilde{x}^s\|^2$, we have
 \begin{align}
  \|x^s_{t+1} - \tilde{x}^s\|^2 & = \|x^s_{t+1} - x^s_t + x^s_t - \tilde{x}^s\|^2 \nonumber \\
  & =  \|x^s_{t+1}-x^s_t\|^2 + 2(x^s_{t+1}-x^s_t)^T(x^s_t - \tilde{x}^s)+ \|x^s_t-\tilde{x}^s\|^2 \nonumber \\
  & \leq \|x^s_{t+1}-x^s_t\|^2 + 2\big( \frac{1}{2\beta}\|x^s_{t+1}-x^s_t)\|^2 + \frac{\beta}{2}\|x^s_t-\tilde{x}^s\|^2\big)+ \|x^s_t-\tilde{x}^s\|^2 \nonumber \\
  & = (1+\frac{1}{\beta})\|x^s_{t+1}-x^s_t\|^2 + (1+\beta)\|x^s_t-\tilde{x}^s\|^2,
 \end{align}
 where $\beta>0$.
 Then we have
 \begin{align}
  \Psi^s_{t+1} & = \mathbb{E}\big[ F(x^s_{t+1}) + c_{t+1}\|x^s_{t+1}-\tilde{x}^s\|^2\big] \nonumber \\
  & \leq \mathbb{E}\big[ F(x^s_{t+1}) + c_{t+1}(1+\frac{1}{\beta})\|x^s_{t+1}-x^s_t\|^2 + c_{t+1}(1+\beta)\|x^s_{t+1}-\tilde{x}^s\|^2\big] \nonumber \\
  & \leq \mathbb{E}\big[ F(x^s_t) +  \big(\frac{3\delta_n L^2\eta}{b} + c_{t+1}(1+\beta)\big)\|x^s_t-\tilde{x}^s\|^2
  + (L-\frac{1}{2\eta})\|\bar{x}^s_{t+1}-x^s_t\|^2 + \big(\frac{L}{2}-\frac{1}{2\eta} + c_{t+1}(1+\frac{1}{\beta})\big)\|x^s_{t+1}-x^s_t\|^2 \nonumber \\
 & \quad  + (1+\frac{6\delta_n}{b})(d+6)^3 L^2\mu^2\eta  + (2 + \frac{12\delta_n}{b})(2d+9)\sigma^2\eta \big]\nonumber \\
 & =  \Psi^s_{t} + (L-\frac{1}{2\eta})\|\bar{x}^s_{t+1}-x^s_t\|^2+ \big(\frac{L}{2}-\frac{1}{2\eta}
  + c_{t+1}(1+\frac{1}{\beta})\big)\|x^s_{t+1}-x^s_t\|^2 \nonumber \\
 & \quad + (1+\frac{6\delta_n}{b})(d+6)^3 L^2\mu^2\eta  + (2 + \frac{12\delta_n}{b})(2d+9)\sigma^2\eta,
 \end{align}
where $c_t = \frac{3\delta_n L^2\eta}{b} + c_{t+1}(1+\beta)$.
Let $c_m = 0$, $\beta=\frac{1}{m}$ and $\eta=\frac{\rho}{L} (0 < \rho < \frac{1}{2})$ , recursing on $t$, we have
\begin{align}
 c_t = \frac{3\delta_n L^2\eta}{b}\frac{(1+\beta)^{m-t}-1}{\beta} & = \frac{3\delta_n L\rho m}{b}\big((1+\frac{1}{m})^{m-t}-1\big) \nonumber \\
 & \leq \frac{3\delta_n L \rho m}{b}(e-1) \leq \frac{6 L \rho m}{b},
\end{align}
where the first inequality holds by $(1+\frac{1}{m})^m$ is an increasing function and $\lim_{m\rightarrow \infty}(1+\frac{1}{m})^m=e$.

It follows that
\begin{align}
 \frac{L}{2} + c_{t+1}(1+\frac{1}{\beta}) &\leq \frac{L}{2} + \frac{6 L \rho m}{b}(1+m) \nonumber \\
 &\leq \frac{L}{2} + \frac{12 L \rho m^2}{b} = (\rho + \frac{24 \rho^2 m^2}{b})\frac{L}{2\rho} \leq \frac{L}{2\rho}=\frac{1}{2\eta},
\end{align}
where the last inequality holds by $\rho + \frac{24\rho^2 m^2}{b} \leq 1$.
Then we obtain
\begin{align} \label{eq:A84}
 \Psi^s_{t+1} \leq \Psi^s_{t} + (L-\frac{1}{2\eta})\|\bar{x}^s_{t+1}-x^s_t\|^2  + (1+\frac{6\delta_n}{b})(d+6)^3 L^2\mu^2\eta  + (2 + \frac{12\delta_n}{b})(2d+9)\sigma^2\eta.
\end{align}
Telescoping inequality \eqref{eq:A84} over $t$ from $0$ to $m-1$, since $x^s_0=x^{s-1}_m=\tilde{x}^{s-1}$ and $x^{s}_m=\tilde{x}^{s}$, we have
\begin{align} \label{eq:A85}
 \frac{1}{m}\sum_{t=1}^m \mathbb{E}\|g_{\eta}(x^s_t)\|^2 \leq \frac{\mathbb{E}[F(\tilde{x}^{s-1})-F(\tilde{x}^s)]}{m\gamma} + (1+\frac{6\delta_n}{b})(d+6)^3\frac{L^2\mu^2\eta}{\gamma}  + (2 + \frac{12\delta_n}{b})(2d+9)\frac{\sigma^2\eta}{\gamma},
\end{align}
where $\gamma = \frac{\eta}{2}-\eta^2L$, and
\begin{align}
 g_{\eta}(x^s_t) = \frac{1}{\eta}\big[ x^s_t - \mbox{Prox}_{\eta\psi}(x^s_t - \eta \nabla f(x^s_t))\big]=\frac{1}{\eta}(x^s_t-\bar{x}^s_{t+1}).
\end{align}

Summing the inequality \eqref{eq:A85} over $s$ from $1$ to $S$, we have
\begin{align}
 \frac{1}{T}\sum_{s=1}^S \sum_{t=1}^m \mathbb{E}\|g_{\eta}(x^s_t)\|^2 &\leq \frac{\mathbb{E}[F(\tilde{x}^0)-F(\tilde{x}^S)]}{T\gamma}
 + (1+\frac{6\delta_n}{b})(d+6)^3 \frac{L^2\mu^2\eta}{\gamma}  + (2 + \frac{12\delta_n}{b})(2d+9)\frac{\sigma^2\eta}{\gamma} \nonumber \\
 & \leq \frac{\mathbb{E}[F(\tilde{x}^0)-F(x_*)]}{T\gamma}
 + (1+\frac{6\delta_n}{b})(d+6)^3 \frac{L^2\mu^2\eta}{\gamma}  + (2 + \frac{12\delta_n}{b})(2d+9)\frac{\sigma^2\eta}{\gamma},
\end{align}
where $x_*$ is an optimal solution of \eqref{eq:1}.

Let $m=[n^{\frac{1}{3}}]$, $b=[n^{\frac{2}{3}}]$ and $\rho = \frac{1}{6}$,
it is easy verified that $\rho + \frac{24\rho^2 m^2}{b} = \frac{5}{6} < 1$ and $\gamma = \frac{\eta}{2}-\eta^2L=\frac{1}{18L}$.
Finally, given $\mu=O(\frac{1}{d\sqrt{T}})$,
we can obtain the above results.

\end{proof}

\subsection{Convergence Analysis of ZO-ProxSAGA-CooSGE}

In this section, we give the convergence analysis of the ZO-ProxSAGA-CooSGE.
First, we give an useful lemma about the upper bound of the variance of estimated gradient.

\begin{lemma} \label{lem:9}
In Algorithm \ref{alg:2} using the CooSGE, given the estimated gradient
$\hat{v}_t = \frac{1}{b} \sum_{i_t\in \mathcal{I}_t} \big(\hat{\nabla} f_{i_t}(x_{t})-\hat{\nabla} f_{i_t}(z^t_{i_t}) \big) + \hat{\phi}_t$
with $\hat{\phi}_t = \frac{1}{n} \sum_{i=1}^n \hat{\nabla} f_i(z^t_i)$,
then the following inequality holds
\begin{align}
 \mathbb{E}\|\hat{v}_t-\nabla f(x_t)\|^2 \leq \frac{2 L^2
 d}{nb} \sum_{i=1}^n \mathbb{E} \|x_t-z^t_i\|^2_2 + \frac{L^2 d^2\mu^2}{2}.
\end{align}
\end{lemma}

\begin{proof}
By the definition of the estimated gradient $\hat{v}_t$, we have
\begin{align} \label{eq:87}
\mathbb{E} \|\hat{v}_t-\nabla f(x^s_t)\|^2 & = \mathbb{E} \|\frac{1}{b} \sum_{i_t\in \mathcal{I}_t} \big(\hat{\nabla} f_{i_t}(x_{t})-\hat{\nabla} f_{i_t}(z^t_{i_t}) \big)
+ \hat{\phi}_t-\nabla f(x_t)\|^2 \nonumber \\
& = \mathbb{E} \|\frac{1}{b} \sum_{i_t\in \mathcal{I}_t} \big(\hat{\nabla} f_{i_t}(x_{t})-\hat{\nabla} f_{i_t}(z^t_{i_t}) \big)
+ \hat{\phi}_t -\hat{\nabla} f(x_t) + \hat{\nabla} f(x_t) -\nabla f(x_t)\|^2 \nonumber \\
& \leq 2 \mathbb{E} \| \frac{1}{b} \sum_{i_t\in \mathcal{I}_t} \big(\hat{\nabla} f_{i_t}(x_{t})-\hat{\nabla} f_{i_t}(z^t_{i_t}) \big)
+ \hat{\phi}_t -\hat{\nabla} f(x_t)\|^2 + 2 \mathbb{E}\|\hat{\nabla} f(x_t) -\nabla f(x_t)\|^2 \nonumber \\
& \leq 2 \mathbb{E} \| \frac{1}{b} \sum_{i_t\in \mathcal{I}_t} \big(\hat{\nabla} f_{i_t}(x_{t})-\hat{\nabla} f_{i_t}(z^t_{i_t}) \big)
- (\hat{\nabla} f(x_t) -\hat{\phi}_t) \|^2 + \frac{L^2d^2\mu^2}{2},
\end{align}
where the second inequality holds by Lemma \ref{lem:5}.
Using
\begin{align}
\mathbb{E}_{\mathcal{I}_t}\big[\frac{1}{b} \sum_{i_t\in \mathcal{I}_t}\hat{\nabla} f_{i_t}(x_{t})-\hat{\nabla} f_{i_t}(z^t_{i_t}) \big] =
\hat{\nabla} f(x_{t})-\frac{1}{n} \sum_{i=1}^n\hat{\nabla} f_{i}(z^t_{i}) = \hat{\nabla} f(x_{t})- \hat{\phi}_t,
\end{align}
we have
\begin{align}
\sum_{i=1}^n \big( \hat{\nabla} f_{i}(x_{t})-\hat{\nabla} f_{i}(z^t_{i_t})
- (\hat{\nabla} f(x_t) -\hat{\phi}_t) \big)
= n \big( \hat{\nabla} f(x_{t}) - \hat{\phi}_t \big) - n \big( \hat{\nabla} f(x_{t}) - \hat{\phi}_t\big)= 0.
\end{align}
It follows that
\begin{align} \label{eq:90}
& \mathbb{E}\| \frac{1}{b} \sum_{i_t\in \mathcal{I}_t} \big(\hat{\nabla} f_{i_t}(x_{t})-\hat{\nabla} f_{i_t}(z^t_{i_t}) \big)
- (\hat{\nabla} f(x_t) -\hat{\phi}_t) \|^2 \nonumber \\
& \leq \frac{1}{bn}\sum_{i=1}^n \mathbb{E} \| \hat{\nabla} f_{i}(x_{t})-\hat{\nabla} f_{i}(z^t_{i})
- (\hat{\nabla} f(x_t) -\hat{\phi}_t) \|^2 \nonumber \\
& = \frac{1}{bn}\sum_{i=1}^n \mathbb{E} \| \hat{\nabla} f_{i}(x_{t})-\hat{\nabla} f_{i}(z^t_{i}) \|^2\
- \|\hat{\nabla} f(x_t) -\hat{\phi}_t\|^2 \nonumber \\
& \leq  \frac{1}{bn}\sum_{i=1}^n \mathbb{E} \| \hat{\nabla} f_{i}(x_{t})-\hat{\nabla} f_{i}(z^t_{i}) \|^2,
\end{align}
where the first inequality holds by Lemmas 4 and 5 in \citep{liu2018zeroth}.
By \eqref{eq:29}, we have
\begin{align} \label{eq:91}
\mathbb{E}\| \hat{\nabla} f_i(x_t) - \hat{\nabla} f_i(z^t_i) \|^2 & = \mathbb{E}\| \sum_{j=1}^d  \frac{\partial f_{i,\mu_j}}{\partial x_{t,j} }e_j
- \frac{\partial f_{i,\mu_j}}{\partial z^t_{i,j}}e_j \|^2 \nonumber \\
& \leq d \sum_{j=1}^d \mathbb{E} \| \frac{\partial f_{i,\mu_j}}{\partial x_{t,j} }
- \frac{\partial f_{i,\mu_j}}{\partial z^t_{i,j}} \|^2 \nonumber \\
& \leq L^2d\sum_{j=1}^d \mathbb{E}\|x_{t,j} - z^t_{i,j}\|^2 = L^2d\|x_t-z^t_i\|^2,
\end{align}
where the first inequality follows the Jensen's inequality, and
the second inequality holds due to that the function $f_{\mu_j}$ is $L$-smooth.

Finally, combining the inequalities \eqref{eq:87}, \eqref{eq:90} and \eqref{eq:91}, we have
the above result.

\end{proof}

Next, based on the above lemma, we study the convergence property of the ZO-ProxSAGA-CooSGE.

\begin{theorem}
 Assume the sequence $\{x_t\}_{t=1}^T$ generated from Algorithm \ref{alg:2} using the CooSGE, and given a positive sequence $\{c_t\}_{t=1}^T$ as follows:
 \begin{align}
  c_t = \frac{ L^2d\eta}{b}+c_{t+1}(1-p)(1+\beta)
 \end{align}
 where $\beta>0$.
 Let $c_T=0$, $\eta =\frac{\rho}{Ld} \ (0< \rho < \frac{1}{2})$, and $b$ satisfies the following inequality:
 \begin{align}
  \frac{32\rho^2n^2}{b^3} + \rho \leq 1,
 \end{align}
 then we have
 \begin{align}
  \mathbb{E} \|g_{\eta}(x_t)\|^2 \leq \frac{\mathbb{E}[F(x_0)-F(x_*)]}{T\gamma} + \frac{L^2 d^2\mu^2\eta}{4\gamma},
 \end{align}
 where $\gamma = \frac{\eta}{2} - L\eta^2$ and $x^*$ is an optimal solution of the problem \eqref{eq:1}.
 Further let $b = [n^{\frac{2}{3}}]$, $\rho = \frac{1}{8}$ and $\mu=O(\frac{1}{\sqrt{dT}})$,
 we obtain
 \begin{align}
 \mathbb{E} \|g_{\eta}(x_t)\|^2 \leq \frac{64dL\mathbb{E}[F(x_0) - F(x_*)]}{3T} + O(\frac{d}{T}).
 \end{align}
\end{theorem}

\begin{proof}
First, we define an iteration by using the full true gradient:
 \begin{align}
  \bar{x}_{t+1} = \mbox{Prox}_{\eta\psi}\big(x_t - \eta\nabla f(x_t)\big),
 \end{align}
 then applying Lemma 2 of \cite{Reddi2016Prox}, we have
 \begin{align} \label{eq:99}
  F(\bar{x}_{t+1}) \leq F(z) + (\frac{L}{2}-\frac{1}{2\eta})\|\bar{x}_{t+1}-x_t\|^2 + (\frac{L}{2}+\frac{1}{2\eta})\|z-x_t\|^2 - \frac{1}{2\eta}\|\bar{x}_{t+1}-z\|^2, \quad \forall z\in \mathbb{R}^d.
 \end{align}
 Since $x_{t+1}=\mbox{Prox}_{\eta\psi}\big(x_t - \eta \hat{v}_t \big)$, we have
 \begin{align} \label{eq:100}
  F(x_{t+1}) \leq F(z) + \langle x_{t+1}-z, \nabla f(x_t)-\hat{v}_t\rangle + (\frac{L}{2}-\frac{1}{2\eta})\|x_{t+1}-x_t\|^2 + (\frac{L}{2}+\frac{1}{2\eta})\|z-x_t\|^2 - \frac{1}{2\eta}\|x_{t+1}-z\|^2.
 \end{align}
 Setting $z=x_t$ in \eqref{eq:99} and $z=\bar{x}_{t+1}$ in \eqref{eq:100},
 then summing them together and taking the expectations, we obtain
 \begin{align} \label{eq:101}
  \mathbb{E}[F(x_{t+1})] \leq & \mathbb{E}\big[ F(x_{t}) + \underbrace{\langle x_{t+1}-\bar{x}_{t+1}, \nabla f(x_t)-\hat{v}_t\rangle}_{T_3} +  (\frac{L}{2}-\frac{1}{2\eta})\|x_{t+1}-x_t\|^2\nonumber \\
  &+ (L-\frac{1}{2\eta})\|\bar{x}_{t+1}-x_t\|^2 - \frac{1}{2\eta}\|x_{t+1}-\bar{x}_{t+1}\|^2 \big].
 \end{align}
 Next, we give an upper bound of the term $T_3$  as follows:
 \begin{align} \label{eq:102}
  T_3 & = \mathbb{E} \langle x_{t+1}-\bar{x}_{t+1}, \nabla f(x_t)-\hat{v}_t\rangle \nonumber \\
  & \leq \frac{1}{2\eta}\mathbb{E}\|x_{t+1}-\bar{x}_{t+1}\|^2 + \frac{\eta}{2}\mathbb{E}\|\nabla f(x_t)-\hat{v}_t\|^2 \nonumber \\
  & \leq \frac{1}{2\eta}\mathbb{E}\|x_{t+1}-\bar{x}_{t+1}\|^2 +  \frac{ L^2
 d\eta}{nb} \sum_{i=1}^n \mathbb{E} \|x_t-z^t_i\|^2_2 + \frac{L^2 d^2\mu^2\eta}{4},
 \end{align}
 where the first inequality holds by Cauchy-Schwarz and Young's inequality and the second inequality holds by
 Lemma \ref{lem:9}.

 Combining \eqref{eq:101} with \eqref{eq:102}, we have
 \begin{align}
  \mathbb{E}[F(x_{t+1})] \leq & \mathbb{E}\big[ F(x_{t}) +  \frac{ L^2
  d\eta}{nb} \sum_{i=1}^n \mathbb{E} \|x_t-z^t_i\|^2_2 + \frac{L^2 d^2\mu^2\eta}{4} +  (\frac{L}{2}-\frac{1}{2\eta})\|x_{t+1}-x_t\|^2\nonumber \\
  &  + (L-\frac{1}{2\eta})\|\bar{x}_{t+1}-x_t\|^2 \big].
 \end{align}

 Next, we define an useful \emph{Lyapunov} function as follows:
 \begin{align}
   \Phi_t = \mathbb{E}\big[ F(x_t) + c_t \frac{1}{n}\sum_{i=1}^n\|x_t-z^t_i\|^2 \big],
 \end{align}
 where $\{c_t\}$ is a nonnegative sequence.
 By the step 7 of Algorithm \ref{alg:2}, we have
 \begin{align} \label{eq:105}
  \frac{1}{n}\sum_{i=1}^n \|x_{t+1}-z^{t+1}_i\|^2 &= \frac{1}{n}\sum_{i=1}^n \big(p\|x_{t+1}-x_{t}\|^2+(1-p)\|x_{t+1}-z^{t}_i\|^2\big) \nonumber \\
   & = \frac{p}{n}\sum_{i=1}^n \|x_{t+1}-x_{t}\|^2 + \frac{1-p}{n}\sum_{i=1}^n\|x_{t+1}-z^{t}_i\|^2 \nonumber \\
   & = p \|x_{t+1}-x_{t}\|^2 + \frac{1-p}{n}\sum_{i=1}^n\|x_{t+1}-z^{t}_i\|^2,
 \end{align}
 where $p$ denotes probability of an index $i$ being in $\mathcal{I}_t$. Here we have
 \begin{align}
  p = 1-(1-\frac{1}{n})^b \geq 1- \frac{1}{1+b/n} = \frac{b/n}{1+b/n} \geq \frac{b}{2n},
 \end{align}
 where the first inequality follows from $(1-a)^b\leq \frac{1}{1+ab}$, and the second inequality holds by
 $b\leq n$.
 Considering the upper bound of $\|x_{t+1}-z^{t}_i\|^2$, we have
 \begin{align} \label{eq:107}
  \|x_{t+1}-z^{t}_i\|^2 & = \|x_{t+1}-x_t+x_t-z^{t}_i\|^2 \nonumber \\
  & =  \|x_{t+1}-x_t\|^2 + 2(x_{t+1}-x_t)^T(x_t-z^{t}_i)+ \|x_t-z^{t}_i\|^2 \nonumber \\
  & \leq \|x_{t+1}-x_t\|^2 + 2\big( \frac{1}{2\beta}\|x_{t+1}-x_t\|^2 + \frac{\beta}{2}\|x_t-z^{t}_i\|^2\big)+ \|x_t-z^{t}_i\|^2 \nonumber \\
  & = (1+\frac{1}{\beta})\|x_{t+1}-x_t\|^2 + (1+\beta)\|x_t-z^{t}_i\|^2,
 \end{align}
 where $\beta>0$.
 Combining \eqref{eq:105} with \eqref{eq:107}, we have
 \begin{align}
 \frac{1}{n}\sum_{i=1}^n \|x_{t+1}-z^{t+1}_i\|^2 \leq (1+\frac{1-p}{\beta})\|x_{t+1}-x_t\|^2 +\frac{(1-p)(1+\beta)}{n}\sum_{i=1}^n\|x_t-z^{t}_i\|^2.
 \end{align}
 It follows that
 \begin{align}
  \Phi_{t+1} & = \mathbb{E}\big[ F(x_{t+1}) + c_{t+1}\frac{1}{n}\sum_{i=1}^n\|x_{t+1}-z^{t+1}_i\|^2\big] \nonumber \\
  & \leq \mathbb{E}\big[ F(x_{t+1}) + c_{t+1}(1+\frac{1-p}{\beta})\|x_{t+1}-x_t\|^2
  + c_{t+1}\frac{(1-p)(1+\beta)}{n}\sum_{i=1}^n\|x_t-z^{t}_i\|^2\big] \nonumber \\
  & \leq \mathbb{E} \big[F(x_t) + \big(\frac{ L^2
  d\eta}{b}+c_{t+1}(1-p)(1+\beta) \big)\frac{1}{n}\sum_{i=1}^n \mathbb{E} \|x_t-z^t_i\|^2_2 + \big( \frac{L}{2}-\frac{1}{2\eta} + c_{t+1}(1+\frac{1-p}{\beta}) \big)\|x_{t+1}-x_t\|^2 \nonumber \\
 & \quad + (L-\frac{1}{2\eta})\|\bar{x}_{t+1}-x_t\|^2 + \frac{L^2 d^2\mu^2\eta}{4} \big] \nonumber \\
 & \leq \Phi_t + \big( \frac{L}{2}-\frac{1}{2\eta} + c_{t+1}(1+\frac{1-p}{\beta}) \big)\|x_{t+1}-x_t\|^2  + (L-\frac{1}{2\eta})\|\bar{x}_{t+1}-x_t\|^2 + \frac{L^2 d^2\mu^2\eta}{4},
 \end{align}
where $c_t = \frac{ L^2d\eta}{b}+c_{t+1}(1-p)(1+\beta)$.

Let $c_T = 0$ and $\beta=\frac{b}{4n}$. Since $(1-p)(1+\beta)=1+\beta-p-p\beta\leq 1+\beta-p$ and $p\geq \frac{b}{2n}$,
we have
\begin{align}
 c_t \leq c_{t+1}(1-\theta) + \frac{ L^2 d\eta}{b},
\end{align}
where $\theta = p-\beta\geq \frac{b}{4n}$.
Recursing on $t$, for $0\leq t \leq T-1$, we have
\begin{align}
 c_t \leq \frac{ L^2 d\eta}{b} \frac{1-\theta^{T-t}}{\theta} \leq \frac{ L^2 d\eta}{b\theta} \leq \frac{ 4nL^2 d\eta}{b^2}.
\end{align}
Using $\eta = \frac{\rho}{dL} \ (0< \rho < \frac{1}{2})$, we obtain $c_t \leq \frac{4n\rho L}{b^2}$.
It follows that
\begin{align}
 c_{t+1}(1+\frac{1-p}{\beta}) + \frac{L}{2} & \leq \frac{4n\rho L}{b^2}( 1 + \frac{4n-2b}{b}) + \frac{L}{2} \nonumber \\
 & = \frac{4n\rho L}{b^2}( \frac{4n}{b} - 1) + \frac{L}{2} \nonumber \\
 & \leq \frac{16\rho Ln^2}{b^3} + \frac{L}{2} = (\frac{32\rho^2n^2}{b^3} + \rho)\frac{L}{2\rho} \leq \frac{L}{2\rho} \leq \frac{1}{2\eta},
\end{align}
where the third inequality holds by $\frac{32\rho^2n^2}{b^3} + \rho \leq 1$.
Thus, we obtain
\begin{align} \label{eq:113}
 \Phi_{t+1} \leq \Phi_t + (L-\frac{1}{2\eta})\|\bar{x}_{t+1}-x_t\|^2 + \frac{L^2 d^2\mu^2\eta}{4}.
\end{align}
Summing the inequality \eqref{eq:113} across all the iterations, we have
\begin{align}
\frac{1}{T}\sum_{t=1}^T(\frac{1}{2\eta}-L)\mathbb{E}\|x_t-\bar{x}_{t+1}\|^2 \leq \frac{\Phi_0-\Phi_T}{T} + \frac{L^2 d^2\mu^2\eta}{4}.
\end{align}

Since $c_T=0$ and $z^i_0=x_0$ for all $i=1,2,\cdots,n$, we have
\begin{align}
 \frac{1}{T}\sum_{t=1}^T (\frac{1}{2\eta}-L)\mathbb{E}\|g_{\eta}(x_t)\|^2 \leq \frac{\mathbb{E}[F(x_0)-F(x_T)]}{T\gamma} + \frac{L^2 d^2\mu^2\eta}{4\gamma},
\end{align}
where $\gamma=\frac{\eta}{2}-L\eta^2$ and
\begin{align}
 g_{\eta}(x_t) = \frac{1}{\eta}\big[x_t-\mbox{Prox}_{\eta \psi}(x_t-\eta\nabla f(x_t))\big] = \frac{1}{\eta}(x_t-\bar{x}_{t+1}).
\end{align}
Given $b=[n^{\frac{2}{3}}]$ and $\rho=\frac{1}{8}$, it is easy verified that $\frac{32\rho^2n^2}{b^3} + \rho = \frac{5}{8}\leq 1$.
and $\gamma=\frac{\eta}{2}-L\eta^2=\frac{1}{16dL}-\frac{1}{64d^2L}\geq \frac{1}{16dL}-\frac{1}{64dL} = \frac{3}{64dL}$.
Finally, let $\mu=O(\frac{1}{\sqrt{dT}})$,
we can obtain the above result.

\end{proof}

\subsection{Convergence Analysis of ZO-ProxSAGA-GauSGE}
In this section, we give the convergence analysis of the ZO-ProxSAGA-GauSGE.
First, we give an useful lemma about the upper bound of the variance of estimated gradient.

\begin{lemma} \label{lem:10}
In Algorithm \ref{alg:2} using GauSGE, given the estimated gradient
$\hat{v}_t = \frac{1}{b} \sum_{i_t\in \mathcal{I}_t} \big(\hat{\nabla} f_{i_t}(x_{t})-\hat{\nabla} f_{i_t}(z^t_{i_t}) \big) + \hat{\phi}_t$
with $\hat{\phi}_t = \frac{1}{n} \sum_{i=1}^n \hat{\nabla} f_i(z^t_i)$,
then the following inequality holds
\begin{align}
 \mathbb{E}\|\hat{v}_t-\nabla f(x_t)\|^2 \leq &  \frac{6L^2}{nb}\sum_{i=1}^n\mathbb{E}\|x_t-z^t_i\|^2 + (4+\frac{24}{b})(2d+9)\sigma^2 + (2+\frac{12}{b})(d+6)^3 L^2\mu^2.
\end{align}
\end{lemma}

\begin{proof}
By the definition of the estimated gradient $\hat{v}_t$, we have
\begin{align} \label{eq:111}
\mathbb{E} \|\hat{v}_t-\nabla f(x_t)\|^2 &= \mathbb{E}\|\frac{1}{b} \sum_{i_t\in \mathcal{I}_t} \big(\hat{\nabla} f_{i_t}(x_{t})-\hat{\nabla} f_{i_t}(z^t_{i_t}) \big)
+ \hat{\phi}_t-\nabla f(x_t)\|^2 \nonumber \\
& = \mathbb{E}\|\frac{1}{b} \sum_{i_t\in \mathcal{I}_t} \big(\hat{\nabla} f_{i_t}(x_{t})-\hat{\nabla} f_{i_t}(z^t_{i_t}) \big)
+ \hat{\phi}_t -\hat{\nabla} f(x_t) + \hat{\nabla} f(x_t) -\nabla f(x_t)\|^2 \nonumber \\
& \leq 2 \mathbb{E}\| \frac{1}{b} \sum_{i_t\in \mathcal{I}_t} \big(\hat{\nabla} f_{i_t}(x_{t})-\hat{\nabla} f_{i_t}(z^t_{i_t}) \big)
+ \hat{\phi}_t -\hat{\nabla} f(x_t)\|^2 + 2\mathbb{E}\|\hat{\nabla} f(x_t) -\nabla f(x_t)\|^2 \nonumber \\
& \leq 2 \mathbb{E}\| \frac{1}{b} \sum_{i_t\in \mathcal{I}_t} \big(\hat{\nabla} f_{i_t}(x_{t})-\hat{\nabla} f_{i_t}(z^t_{i_t}) \big)
- (\hat{\nabla} f(x_t) -\hat{\phi}_t) \|^2  + 4(2d+9)\mathbb{E}\|\nabla f(x_t)\|^2 + 2\mu^2L^2(d+6)^3,
\end{align}
where the second inequality holds by Lemma \ref{lem:6}. Using
\begin{align}
\mathbb{E}_{\mathcal{I}_t}\big[\frac{1}{b} \sum_{i_t\in \mathcal{I}_t}\hat{\nabla} f_{i_t}(x_{t})-\hat{\nabla} f_{i_t}(z^t_{i_t}) \big] =
\hat{\nabla} f(x_{t})-\frac{1}{n} \sum_{i=1}^n\hat{\nabla} f_{i}(z^t_{i}) = \hat{\nabla} f(x_{t})- \hat{\phi}_t,
\end{align}
then we have
\begin{align}
\sum_{i=1}^n \big( \hat{\nabla} f_{i}(x_{t})-\hat{\nabla} f_{i}(z^t_{i_t})
- (\hat{\nabla} f(x_t) -\hat{\phi}_t) \big)
= n \big( \hat{\nabla} f(x_{t}) - \hat{\phi}_t \big) - n \big( \hat{\nabla} f(x_{t}) - \hat{\phi}_t\big)= 0.
\end{align}
It follows that
\begin{align} \label{eq:114}
& \mathbb{E}\| \frac{1}{b} \sum_{i_t\in \mathcal{I}_t} \big(\hat{\nabla} f_{i_t}(x_{t})-\hat{\nabla} f_{i_t}(z^t_{i_t}) \big)
- (\hat{\nabla} f(x_t) -\hat{\phi}_t) \|^2 \nonumber \\
& \leq \frac{1}{bn}\sum_{i=1}^n \mathbb{E} \| \hat{\nabla} f_{i}(x_{t})-\hat{\nabla} f_{i}(z^t_{i})
- (\hat{\nabla} f(x_t) -\hat{\phi}_t) \|^2 \nonumber \\
& = \frac{1}{bn}\sum_{i=1}^n \mathbb{E} \| \hat{\nabla} f_{i}(x_{t})-\hat{\nabla} f_{i}(z^t_{i}) \|^2\ - \|\hat{\nabla} f(x_t) -\hat{\phi}_t\|^2 \nonumber \\
& \leq  \frac{1}{bn}\sum_{i=1}^n \mathbb{E} \| \hat{\nabla} f_{i}(x_{t})-\hat{\nabla} f_{i}(z^t_{i}) \|^2,
\end{align}
where the first inequality holds by Lemmas 4 and 5 in \citep{liu2018zeroth}.
By \eqref{eq:29}, we have
\begin{align} \label{eq:115}
\mathbb{E}\| \hat{\nabla} f_i(x_t) - \hat{\nabla} f_i(z^t_i) \|^2 & = \mathbb{E}\| \hat{\nabla} f_i(x_t) - \nabla f_i (x_t) +  \nabla f_i (x_t)
- \nabla f_i(z^t_i) + \nabla f_i(z^t_i) - \hat{\nabla}f_i(z^t_i) \|^2 \nonumber \\
& \leq 3 \mathbb{E}\| \hat{\nabla} f_i(x_t) - \nabla f_i (x_t) \|^2 + 3 \|\nabla f_i (x_t)
- \nabla f_i(z^t_i)\|^2 + 3 \mathbb{E}\| \nabla f_i(z^t_i) - \hat{\nabla}f_i(z^t_i) \|^2\nonumber \\
& \leq 6(2d+9)\big( \|\nabla f_i(x_t)\|^2 + \|\nabla f_i(z^t_i)\|^2 \big) + 3L^2\|x_t-z^t_i\|^2 + 6L^2\mu^2(d+6)^3 ,
\end{align}
where the first inequality follows the Jensen's inequality, and
the second inequality holds by the Lemma \ref{lem:6}.
Finally, combining the inequalities \eqref{eq:111}, \eqref{eq:114} and \eqref{eq:115}, we obtain the above result.

\end{proof}

Next, based on the above lemma, we study the convergence property of the ZO-ProxSAGA-GauSGE.

\begin{theorem}
 Assume the sequence $\{x_t\}_{t=1}^T$ generated from Algorithm \ref{alg:2} using the GauSGE, and given a positive sequence $\{c_t\}_{t=1}^T$ as follows:
 \begin{align}
  c_t = \frac{ 3L^2\eta}{b}+c_{t+1}(1-p)(1+\beta),
 \end{align}
 where $\beta>0$. Let $c_T=0$, $\eta = \frac{\rho}{L} (0< \rho < \frac{1}{2})$ and $b$ satisfies the following inequality:
 \begin{align}
  \frac{96\rho^2n^2}{b^3} + \rho \leq 1,
 \end{align}
 then we have
 \begin{align}
  \mathbb{E} \|g_{\eta}(x_t)\|^2 \leq  \frac{\mathbb{E}[F(x_0)-F(x_*)]}{T\gamma} + \frac{(2+\frac{12}{b})(2d+9)\sigma^2\eta}{\gamma} + \frac{(1+\frac{6}{b})(d+6)^3 L^2\mu^2\eta}{\gamma},
 \end{align}
 where $\gamma=\frac{1}{2\eta}-L\eta^2$ and $x^*$ is an optimal solution of the problem \eqref{eq:1}. Further given $b=[n^{\frac{2}{3}}]$,
 $\rho = \frac{1}{12}$ and $\mu=O(\frac{1}{d\sqrt{T}})$, we have
 \begin{align}
 \mathbb{E} \|g_{\eta}(x_t)\|^2 \leq  & \frac{144L\mathbb{E}[F(x_0)-F(x_*)]}{5T}+ O(\frac{d}{T})+ O(d\sigma^2).
 \end{align}
\end{theorem}

\begin{proof}
This proof is the similar to the proof of Theorem \ref{th:3}.
We begin with defining an iteration by using the full true gradient: 
 \begin{align}
  \bar{x}_{t+1} = \mbox{Prox}_{\eta\psi}\big(x_t - \eta\nabla f(x_t)\big).
 \end{align}
Using Lemma 2 of \cite{Reddi2016Prox}, we have for any $z\in \mathbb{R}^d$
 \begin{align} \label{eq:128}
  F(\bar{x}_{t+1}) \leq F(z) + (\frac{L}{2}-\frac{1}{2\eta})\|\bar{x}_{t+1}-x_t\|^2 + (\frac{L}{2}+\frac{1}{2\eta})\|z-x_t\|^2 - \frac{1}{2\eta}\|\bar{x}_{t+1}-z\|^2.
 \end{align}
 Since $x_{t+1}=\mbox{Prox}_{\eta\psi}\big(x_t - \eta \hat{v}_t \big)$, similarly, we have
 \begin{align} \label{eq:129}
  F(x_{t+1}) \leq F(z) + \langle x_{t+1}-z, \nabla f(x_t)-\hat{v}_t\rangle + (\frac{L}{2}-\frac{1}{2\eta})\|x_{t+1}-x_t\|^2 + (\frac{L}{2}+\frac{1}{2\eta})\|z-x_t\|^2 - \frac{1}{2\eta}\|x_{t+1}-z\|^2.
 \end{align}
 Setting $z=x_t$ in \eqref{eq:128} and $z=\bar{x}_{t+1}$ in \eqref{eq:129},
 then summing them together and taking the expectations, we obtain
 \begin{align} \label{eq:130}
  \mathbb{E}[F(x_{t+1})] \leq & \mathbb{E}\big[ F(x_{t}) + \underbrace{\langle x_{t+1}-\bar{x}_{t+1}, \nabla f(x_t)-\hat{v}_t\rangle}_{T_4} +  (\frac{L}{2}-\frac{1}{2\eta})\|x_{t+1}-x_t\|^2 \nonumber \\
  & + (L-\frac{1}{2\eta})\|\bar{x}_{t+1}-x_t\|^2 - \frac{1}{2\eta}\|x_{t+1}-\bar{x}_{t+1}\|^2 \big].
 \end{align}
Next, we give an upper bound of the term $T_4$ as follows:
 \begin{align} \label{eq:131}
  T_4 & = \mathbb{E} \langle x_{t+1}-\bar{x}_{t+1}, \nabla f(x_t)-\hat{v}_t\rangle \nonumber \\
  & \leq \frac{1}{2\eta}\mathbb{E}\|x_{t+1}-\bar{x}_{t+1}\|^2 + \frac{\eta}{2}\mathbb{E}\|\nabla f(x_t)-\hat{v}_t\|^2 \nonumber \\
  & \leq \frac{1}{2\eta}\mathbb{E}\|x_{t+1}-\bar{x}_{t+1}\|^2 + \frac{3L^2\eta}{nb}\sum_{i=1}^n\mathbb{E}\|x_t-z^t_i\|^2 + (2+\frac{12}{b})(2d+9)\eta\sigma^2 + (1+\frac{6}{b})(d+6)^3 L^2\mu^2\eta,
 \end{align}
 where the first inequality holds by Cauchy-Schwarz and Young's inequality and the second inequality holds by
 Lemma \ref{lem:10}.
Combining \eqref{eq:130} with \eqref{eq:131}, we have
 \begin{align}
  \mathbb{E}[F(x_{t+1})] \leq & \mathbb{E}\big[ F(x_{t}) +  \frac{3L^2\eta}{nb}\sum_{i=1}^n\mathbb{E}\|x_t-z^t_i\|^2 + (2+\frac{12}{b})(2d+9)\eta\sigma^2 + (1+\frac{6}{b})(d+6)^3 L^2\mu^2\eta \nonumber \\
  & +  (\frac{L}{2}-\frac{1}{2\eta})\|x_{t+1}-x_t\|^2 + (L-\frac{1}{2\eta})\|\bar{x}_{t+1}-x_t\|^2 \big].
 \end{align}

 Next, we define an useful \emph{Lyapunov} function as follows:
 \begin{align}
   \Omega_t = \mathbb{E}\big[ F(x_t) + c_t \frac{1}{n}\sum_{i=1}^n\|x_t-z^t_i\|^2 \big],
 \end{align}
 where $\{c_t\}$ is a nonnegative sequence.
 By the step 7 of Algorithm \ref{alg:2}, we have
 \begin{align} \label{eq:134}
  \frac{1}{n}\sum_{i=1}^n \|x_{t+1}-z^{t+1}_i\|^2 &= \frac{1}{n}\sum_{i=1}^n \big(p\|x_{t+1}-x_{t}\|^2+(1-p)\|x_{t+1}-z^{t}_i\|^2\big) \nonumber \\
   & = \frac{p}{n}\sum_{i=1}^n \|x_{t+1}-x_{t}\|^2 + \frac{1-p}{n}\sum_{i=1}^n\|x_{t+1}-z^{t}_i\|^2 \nonumber \\
   & = p \|x_{t+1}-x_{t}\|^2 + \frac{1-p}{n}\sum_{i=1}^n\|x_{t+1}-z^{t}_i\|^2,
 \end{align}
 where $p$ denotes probability of an index $i$ being in $\mathcal{I}_t$. Here we have
 \begin{align}
  p = 1-(1-\frac{1}{n})^b \geq 1- \frac{1}{1+b/n} = \frac{b/n}{1+b/n} \geq \frac{b}{2n},
 \end{align}
 where the first inequality follows from $(1-a)^b\leq \frac{1}{1+ab}$, and the second inequality holds by
 $b\leq n$.
 Considering the upper bound of $\|x_{t+1}-z^{t}_i\|^2$, we have
 \begin{align} \label{eq:136}
  \|x_{t+1}-z^{t}_i\|^2 & = \|x_{t+1}-x_t+x_t-z^{t}_i\|^2 \nonumber \\
  & =  \|x_{t+1}-x_t\|^2 + 2(x_{t+1}-x_t)^T(x_t-z^{t}_i)+ \|x_t-z^{t}_i\|^2 \nonumber \\
  & \leq \|x_{t+1}-x_t\|^2 + 2\big( \frac{1}{2\beta}\|x_{t+1}-x_t\|^2 + \frac{\beta}{2}\|x_t-z^{t}_i\|^2\big)+ \|x_t-z^{t}_i\|^2 \nonumber \\
  & = (1+\frac{1}{\beta})\|x_{t+1}-x_t\|^2 + (1+\beta)\|x_t-z^{t}_i\|^2,
 \end{align}
 where $\beta>0$.
 Combining \eqref{eq:134} with \eqref{eq:136}, we have
 \begin{align}
 \frac{1}{n}\sum_{i=1}^n \|x_{t+1}-z^{t+1}_i\|^2 \leq (1+\frac{1-p}{\beta})\|x_{t+1}-x_t\|^2 +\frac{(1-p)(1+\beta)}{n}\sum_{i=1}^n\|x_t-z^{t}_i\|^2.
 \end{align}
 It follows that
 \begin{align}
  \Omega_{t+1} & = \mathbb{E}\big[ F(x_{t+1}) + c_{t+1}\frac{1}{n}\sum_{i=1}^n\|x_{t+1}-z^{t+1}_i\|^2\big] \nonumber \\
  & \leq \mathbb{E}\big[ F(x_{t+1}) + c_{t+1}(1+\frac{1-p}{\beta})\|x_{t+1}-x_t\|^2
  + c_{t+1}\frac{(1-p)(1+\beta)}{n}\sum_{i=1}^n\|x_t-z^{t}_i\|^2\big] \nonumber \\
  & \leq \mathbb{E} \big[F(x_t) + \big(\frac{3L^2\eta}{b}+c_{t+1}(1-p)(1+\beta) \big)\frac{1}{n}\sum_{i=1}^n \mathbb{E} \|x_t-z^t_i\|^2_2 + \big( \frac{L}{2}-\frac{1}{2\eta} + c_{t+1}(1+\frac{1-p}{\beta}) \big)\|x_{t+1}-x_t\|^2 \nonumber \\
 & \quad + (L-\frac{1}{2\eta})\|\bar{x}_{t+1}-x_t\|^2 + (2+\frac{12}{b})(2d+9)\eta\sigma^2 + (1+\frac{6}{b})(d+6)^3 L^2\mu^2\eta \big] \nonumber \\
 & \leq \Omega_t + \big( \frac{L}{2}-\frac{1}{2\eta} + c_{t+1}(1+\frac{1-p}{\beta}) \big)\|x_{t+1}-x_t\|^2  + (L-\frac{1}{2\eta})\|\bar{x}_{t+1}-x_t\|^2 + (2+\frac{12}{b})(2d+9)\eta\sigma^2 + (1+\frac{6}{b})(d+6)^3 L^2\mu^2\eta,
 \end{align}
where $c_t = \frac{ 3L^2\eta}{b}+c_{t+1}(1-p)(1+\beta)$.

Let $c_T = 0$ and $\beta=\frac{b}{4n}$. Since $(1-p)(1+\beta)=1+\beta-p-p\beta\leq 1+\beta-p$ and $p\geq \frac{b}{2n}$,
it follows that
\begin{align}
 c_t \leq c_{t+1}(1-\theta) + \frac{ 3L^2\eta}{b},
\end{align}
where $\theta = p-\beta\geq \frac{b}{4n}$.
Then recursing on $t$, for $0\leq t \leq T-1$, we have
\begin{align}
 c_t \leq \frac{ 3L^2 \eta}{b} \frac{1-\theta^{T-t}}{\theta} \leq \frac{ 3L^2\eta}{b\theta} \leq \frac{ 12nL^2\eta}{b^2}.
\end{align}
Let $\eta = \frac{\rho}{L} \ (0< \rho < \frac{1}{2})$, we have $c_t \leq \frac{12n\rho L}{b^2}$.
It follows that
\begin{align}
 c_{t+1}(1+\frac{1-p}{\beta}) + \frac{L}{2} & \leq \frac{12n\rho L}{b^2}( 1 + \frac{4n-2b}{b}) + \frac{L}{2} \nonumber \\
 & = \frac{12n\rho L}{b^2}( \frac{4n}{b} - 1) + \frac{L}{2} \nonumber \\
 & \leq \frac{48\rho Ln^2}{b^3} + \frac{L}{2} = (\frac{96\rho^2n^2}{b^3} + \rho)\frac{L}{2\rho} \leq \frac{L}{2\rho} = \frac{1}{2\eta},
\end{align}
where the third inequality holds by $\frac{96\rho^2n^2}{b^3} + \rho \leq 1$.
Thus, we obtain
\begin{align} \label{eq:142}
 \Omega_{t+1} \leq \Omega_t + (L-\frac{1}{2\eta})\|\bar{x}_{t+1}-x_t\|^2  + (2+\frac{12}{b})(2d+9)\eta\sigma^2 + (1+\frac{6}{b})(d+6)^3 L^2\mu^2\eta.
\end{align}
Summing the inequality \eqref{eq:142} across all the iterations, we have
\begin{align}
\frac{1}{T}\sum_{t=1}^T(\frac{1}{2\eta}-L)\mathbb{E}\|x_t-\bar{x}_{t+1}\|^2 \leq \frac{\Omega_0-\Omega_T}{T}  + (2+\frac{12}{b})(2d+9)\eta\sigma^2 + (1+\frac{6}{b})(d+6)^3 L^2\mu^2\eta.
\end{align}

Since $c_T=0$ and $z^i_0=x_0$ for all $i=1,2,\cdots,n$, we have
\begin{align}
 \frac{1}{T}\sum_{t=1}^T (\frac{1}{2\eta}-L)\mathbb{E}\|g_{\eta}(x_t)\|^2 \leq \frac{\mathbb{E}[F(x_0)-F(x_T)]}{T\gamma}  + (2+\frac{12}{b})(2d+9)\eta\sigma^2 + (1+\frac{6}{b})(d+6)^3 L^2\mu^2\eta,
\end{align}
where $\gamma=\frac{\eta}{2}-L\eta^2$ and
\begin{align}
 g_{\eta}(x_t) = \frac{1}{\eta}\big[x_t-\mbox{Prox}_{\eta \psi}(x_t-\eta\nabla f(x_t))\big] = \frac{1}{\eta}(x_t-\bar{x}_{t+1}).
\end{align}
Given $b=[n^{\frac{2}{3}}]$ and $\rho=\frac{1}{12}$, it is easy verified that $\frac{96\rho^2n^2}{b^3} + \rho = \frac{3}{4}\leq 1$, and
$\gamma=\frac{\eta}{2}-L\eta^2  = \frac{5}{144L}$.
Finally, let $\mu=O(\frac{1}{d\sqrt{T}})$,
we can obtain the above result.

\end{proof}

\end{appendices}

\end{onecolumn}

\end{document}